\documentclass{article}

\usepackage{url}

\usepackage{microtype}
\usepackage{graphicx}
\usepackage{subcaption}
\usepackage{booktabs} 

\usepackage[hypertexnames=false]{hyperref}
\usepackage{todonotes}
\usepackage{adjustbox}

\usepackage[accepted]{icml2026}

\usepackage{amsmath}
\usepackage{amssymb}
\usepackage{mathtools}
\usepackage{amsthm}

\usepackage{tcolorbox}
\usepackage{pifont}
\usepackage{xcolor}
\definecolor{mydarkgreen}{RGB}{39,130,67}
\definecolor{mydarkorange}{RGB}{236,147,14}
\definecolor{mydarkred}{RGB}{236,147,14}
\definecolor{blue}{RGB}{0,0,255}

\DeclareSymbolFont{extraup}{U}{zavm}{m}{n}
\DeclareMathSymbol{\varheart}{\mathalpha}{extraup}{86}
\DeclareMathSymbol{\vardiamond}{\mathalpha}{extraup}{87}





\DeclareSymbolFont{extraup}{U}{zavm}{m}{n}
\DeclareMathSymbol{\varheart}{\mathalpha}{extraup}{86}
\DeclareMathSymbol{\vardiamond}{\mathalpha}{extraup}{87}

\hypersetup{
  colorlinks   = true, 
  urlcolor     = blue, 
  linkcolor    = {red!75!black}, 
  citecolor   = {blue!50!black} 
}

\renewcommand{\eqref}[1]{(\ref{#1})}

\allowdisplaybreaks

\usepackage{graphicx}
\usepackage{apptools}
\usepackage[flushleft]{threeparttable}
\usepackage{array,booktabs,makecell}
\usepackage{multirow}
\usepackage{nicefrac}
\usepackage{amsthm}
\usepackage{amsmath,amsfonts,bm}
\usepackage{wrapfig}
\usepackage{caption}
\usepackage{siunitx}
\usepackage{thm-restate}
\usepackage{nccmath}
\usepackage{empheq}
\usepackage{bbm}
\usepackage{tabularx}

\usepackage{algorithmic}
\usepackage{algorithm}

\usepackage{color}
\usepackage{colortbl}
\definecolor{bgcolor}{rgb}{0.76,0.88,0.50}
\definecolor{bgcolor0}{rgb}{0.93,0.99,1}
\definecolor{bgcolor1}{rgb}{0.8,1,1}
\definecolor{bgcolor2}{rgb}{0.8,1,0.8}
\definecolor{bgcolor3}{rgb}{0.50,0.90,0.50}
\usepackage{tcolorbox}
\usepackage{pifont}

\usepackage[scaled=0.86]{helvet}
\newcommand{\algname}[1]{{#1}}

\newcommand{\norm}[1]{\left\| #1 \right\|}

\newcommand{\inp}[2]{\left\langle#1,#2\right\rangle} 

\newcommand{\R}{\mathbb{R}} 
\newcommand{\N}{\mathbb{N}} 

\newcommand{\cO}{\mathcal{O}}

\theoremstyle{plain}
\newtheorem{theorem}{Theorem}[section]
\newtheorem*{theorem*}{Theorem}

\newtheorem{lemma}[theorem]{Lemma}
\newtheorem{corollary}[theorem]{Corollary}
\theoremstyle{definition}
\newtheorem{definition}[theorem]{Definition}
\newtheorem{assumption}[theorem]{Assumption}

\theoremstyle{remark}

\newcommand{\eqdef}{:=}

\makeatletter
\newcommand{\vast}{\bBigg@{4}}

\icmltitlerunning{Near-Optimal Convergence of Accelerated Gradient Methods under Generalized Smoothness}

\begin{document}

\twocolumn[
  \icmltitle{Near-Optimal Convergence of Accelerated Gradient Methods \\ under Generalized and $(L_0, L_1)$--Smoothness}



  \icmlsetsymbol{equal}{*}

  \begin{icmlauthorlist}
    \icmlauthor{Alexander Tyurin}{yyy,xxx}
  \end{icmlauthorlist}

  \icmlaffiliation{yyy}{AXXX, Moscow, Russia}
  \icmlaffiliation{xxx}{Applied AI Institute, Moscow, Russia}

  \icmlcorrespondingauthor{Alexander Tyurin}{alexandertiurin@gmail.com}

  \icmlkeywords{Machine Learning, ICML}

  \vskip 0.3in
]



\printAffiliationsAndNotice{}  

\begin{abstract}
We study first‐order methods for convex optimization problems with functions \(f\) satisfying the recently proposed $\ell$-smoothness condition  
$\norm{\nabla^{2}f(x)} \le \ell\left(\norm{\nabla f(x)}\right),$  
which generalizes the \(L\)--smoothness and $(L_{0},L_{1})$--smoothness.  
While accelerated gradient descent (\algname{AGD}) is known to reach the optimal complexity  
$\cO(\sqrt{L} R / \sqrt{\varepsilon})$ under $L$--smoothness, where $\varepsilon$ is an error tolerance and $R$ is the distance between a starting and an optimal point, existing extensions to $\ell$--smoothness either incur extra dependence on the initial gradient, suffer exponential factors in $L_{1} R$, or require costly auxiliary sub-routines, leaving open whether an \algname{AGD}‐type $\cO(\sqrt{\ell(0)} R / \sqrt{\varepsilon})$ rate is possible for small--$\varepsilon$, even in the $(L_{0},L_{1})$-smoothness case. We resolve this open question. Developing new proof techniques, we achieve $\cO(\sqrt{\ell(0)} R / \sqrt{\varepsilon})$ oracle complexity for small--$\varepsilon$ and virtually any $\ell$.  
For instance, for \((L_{0},L_{1})\)-smoothness, our bound $\cO(\sqrt{L_0} R / \sqrt{\varepsilon})$ is provably optimal in the small-$\varepsilon$ regime and removes all non-constant multiplicative factors present in prior accelerated algorithms.
\end{abstract}

\section{Introduction}
We focus on optimization problems
\begin{align}
\label{eq:main_problem}
\textstyle \min\limits_{x \in \R^d} f(x),
\end{align}
where $f\,:\,\R^d \to \R \cup \{\infty\}$ is a convex function. We aim to find an $\varepsilon$-solution, $\bar{x} \in \R^d,$ such that $f(\bar{x}) - \inf_{x \in \R^d} f(x) \leq \varepsilon.$ We define $\mathcal{X} = \left\{x \in \R^d\,|\, f(x) < \infty\right\},$ and assume that $\mathcal{X}$ is an open and $d$--dimensional convex set, $f$ is smooth on $\mathcal{X},$ and continuous on the closure of $\mathcal{X}.$ We define $R \eqdef \norm{x^0 - x^*},$ where $x^0 \in \mathcal{X}$ is a starting point of numerical methods.

Under the $L$--smoothness assumption, i.e., $\norm{\nabla f(x) - \nabla f(y)} \leq L \norm{x - y}$ or $\norm{\nabla^2 f(x)} \leq L$ for all $x, y \in \mathcal{X},$ the problem is well studied. In particular, it is known that one can find an $\varepsilon$-solution after $\cO\left(\nicefrac{\sqrt{L} R}{\sqrt{\varepsilon}}\right)$ gradient calls using the fast/accelerated gradient descent method (\algname{AGD}) by \citet{nesterov1983method}, which is also optimal \citep{nemirovskij1983problem,nesterov2018lectures}. This result improves the oracle complexity $\cO\left(\nicefrac{L R^2}{\varepsilon}\right)$ (\# of gradient calculations) of gradient descent (\algname{GD}).

In this work, we investigate the modern $\ell$--smoothness assumption \citep{li2024convex}, which states that $\norm{\nabla^2 f(x)} \leq \ell(\norm{\nabla f(x)})$ for all $x \in \mathcal{X}$ (see Assumption~\ref{ass:gen_smooth}), where $\ell$ is any non-decreasing, positive, locally Lipschitz function. This generalizes the classical $L$--smoothness assumption, which corresponds to the special case $\ell(s) = L$.
An important example of this framework is the $(L_0, L_1)$--smoothness condition \citep{zhang2019gradient}, obtained by setting $\ell(s) = L_0 + L_1 s$, which yields $\norm{\nabla^2 f(x)} \leq L_0 + L_1 \norm{\nabla f(x)}$ for all $x \in \mathcal{X}.$

There are many functions that are captured by $\ell$--smoothness but not by $L$--smoothness. For instance, $f(x) = x^p$ for $p > 2,$ $f(x) = e^x,$ and $f(x) = -\log x$ all satisfy $\ell$--smoothness (with a proper $\ell$) but violate the standard $L$--smoothness condition \citep{li2024convex}. Moreover, there is growing evidence that $\ell$--smoothness is a more appropriate assumption for modern machine learning problems \citep{zhang2019gradient,chen2023generalized,cooper2024theoretical,tyurin2024toward}.

Despite the recent significant interest in $\ell$--smoothness, to the best of our knowledge, one important \emph{open problem} remains:
\begin{quote}
Under $\ell$--smoothness and $(L_0, L_1)$--smoothness, for a small $\varepsilon,$ is it possible to design a method with oracle complexity $\cO\big(\nicefrac{\sqrt{\ell(0)} R}{\sqrt{\varepsilon}}\big)$ and $\cO\big(\nicefrac{\sqrt{L_0} R}{\sqrt{\varepsilon}}\big),$ respectively?
\end{quote}
In this work, using new proof techniques, we provide an \emph{affirmative answer} to this question by developing new approaches that work for all $\varepsilon > 0$ and achieve the optimal complexity under $(L_0, L_1)$–smoothness for small $\varepsilon$.
\begin{table*}[t]
  \caption{Convergence rates for various \algname{AGD} methods up to constant factors.
  Abbreviations: $R \eqdef \norm{x^0 - x^*},$ $\varepsilon =$ error tolerance, $x^0$ is a starting point, $\Delta \eqdef f(x^0) - f(x^*),$ $M_{R}$ is defined in Theorem~\ref{thm:convex_increasing_super}.}
  \label{table:complexities_convex}
  \centering 
  \begin{adjustbox}{width=2\columnwidth,center}
  \begin{threeparttable}
  \bgroup
  \def\arraystretch{3.4}
  \begin{tabular}[t]{c|ccc}
   \Xhline{0.5pt}
   \hline
   \bf Setting & \bf Oracle Complexity & \bf References & \bf \makecell{Required Input Parameters \\ to Algorithm} \\
    \Xhline{0.5pt}
     \hline
     \makecell{$L$--Smoothness} & ${\color{mydarkgreen} \frac{\sqrt{L} R}{\sqrt{\varepsilon}}}$ & \citep{nesterov1983method} & $L$ \\
     \Xhline{0.5pt}
     \hline
      \multirow{4}{*}{\makecell{$(L_0, L_1)$--Smoothness}} 
     & $\geq$\tnote{\color{blue}(a)}$\quad {\color{red}\left(L_1^2 R^2 + \frac{L_1^2 \Delta}{L_0} + 1\right)} \sqrt{\frac{\Delta + L_0 R^2}{\varepsilon}}$
     & \citep{li2024convex} & $L_0, L_1, R, \Delta$ \\ \cline{2-4}
     & ${\color{red}\sqrt{1 + L_1 R \exp(L_1 R)}} \times \frac{\sqrt{L_0} R}{\sqrt{\varepsilon}}$ & \citep{gorbunov2024methods} & $L_0, L_1$\\
     \cline{2-4}
     & \makecell{${\color{red} \nu} \times \left(\frac{\sqrt{L_0} R}{\sqrt{\varepsilon}} + (L_1 R)^{2/3} \log\left(\frac{\Delta}{\varepsilon}\right)\right),$ \\ where $\nu$ is not a universal constant\textsuperscript{\color{blue}(b)} and \\ may depend on parameters of $f,$$\varepsilon,$ and $R$ } & \citep{vankov2024optimizing} & \makecell{$L_0, L_1,$ \\ params for auxiliary problem \\ (e.g., \# of inner iterations)} \\
     \cline{2-4}
     & ${\color{mydarkgreen} \frac{\sqrt{L_0} R}{\sqrt{\varepsilon}}} + L_1 R \log\Big(\min \Big\{\frac{L_1^2 \Delta}{L_0}, \frac{\Delta}{\varepsilon}\Big\}\Big)$ & \makecell{Sec.~\ref{sec:example_1}, \ref{sec:example_2}, \\ or Thm.~\ref{thm:convex_increasing_better_rate_refine} (\textbf{new})} & \makecell{$L_0, L_1, R, \Delta$ \\ (semi-adaptive to $R, \Delta$)} \\
     \Xhline{0.5pt}
     \hline
     \multirow{2.0}{*}{\makecell{General result \\ with any $\ell$}} & $\geq$\tnote{\color{blue}(a)}$\quad \frac{{\color{red}\sqrt{\ell\left(\norm{\nabla f(x^0)}\right)}} R}{\sqrt{\varepsilon}}$ & \makecell{\citep{li2024convex}} & $L_0, L_1, R, \Delta$ \\ \cline{2-4}
     & \makecell{${\color{mydarkgreen} \frac{\sqrt{\ell(0)} R}{\sqrt{\varepsilon}}} + T,$ \\ where $T$ does not depend on $\varepsilon$} & Corollary~\ref{cor:general} (\textbf{new}) & $L_0, L_1, R, \Delta, M_{R}$ \\
    \Xhline{0.5pt}
    \hline
    \end{tabular}
    \egroup
    \begin{tablenotes}
\item [{\color{blue}(a)}] The obtained result is not better than this complexity. For the $(L_0,L_1)$--smoothness, the explicit formula is derived in \citep{vankov2024optimizing}. The norm $\norm{\nabla f(x^0)}$ can be exponentially large in the parameters $L_1$ and $R$ (e.g., take the function $\exp(L_1 x)$).
\item [{\color{blue}(b)}] The quantity $\nu$ arises because the accelerated method of \citet{vankov2024optimizing} solves an auxiliary problem at each iteration, requiring $\nu$ additional oracle calls. Their analysis also assumes that this auxiliary problem is solved exactly, which may not hold in practice and theory, and the stability to solution errors is unclear. In contrast, our method and the methods of \cite{li2024convex,gorbunov2024methods} compute only one gradient per iteration.
\end{tablenotes}
\end{threeparttable}
\end{adjustbox}
\end{table*}

\subsection{Related work}
\label{sec:related_work}
\textbf{Nonconvex optimization with $(L_0, L_1)$--smoothness.} 
While we focus on convex problems, we now recall the modern results in the non-convex setting. \citet{zhang2019gradient} is the seminal work that considers $(L_0, L_1)$--smoothness. They developed a clipped version of \algname{GD} that finds an $\varepsilon$--stationary point after $\cO\left(\nicefrac{L_0 \Delta}{\varepsilon} + \nicefrac{L_1^2 \Delta}{L_0}\right)$ iterations\footnote{An $\varepsilon$--stationary point is a point $\bar{x}$ such that $\norm{\nabla f(\bar{x})}^2 \leq \varepsilon;$ $\Delta \eqdef f(x^0) - f^*,$ where $x^0$ is a starting point of numerical methods.}. There are many subsequent works on $(L_0, L_1)$--smoothness, including \citep{crawshaw2022robustness,chen2023generalized,wang2023convergence,koloskova2023revisiting,li2024convex,li2024convergence,hubler2024parameter,vankov2024optimizing}. Under $(L_0, L_1)$--smoothness, the state-of-the-art theoretical oracle complexity $\cO\left(\nicefrac{L_0 \Delta}{\varepsilon} + \nicefrac{L_1 \Delta}{\sqrt{\varepsilon}}\right)$ was proved by \citet{vankov2024optimizing}. 

\textbf{Nonconvex optimization with $\ell$--smoothness.} 
The paper by \citet{li2024convex} is the seminal work that introduces the $\ell$--smoothness assumption. In their version of \algname{GD}, the result depends on $\ell(\norm{\nabla f(x^0)}) / \varepsilon$ and requires $\ell$ to grow more slowly than $s^2$. Subsequently, \citet{tyurin2024toward} improved their oracle
complexity and provided the current state-of-the-art complexity. For instance, under $(\rho, L_0, L_1)$--smoothness, i.e., $\norm{\nabla^2 f(x)} \leq L_0 + L_1 \norm{\nabla f(x)}^{\rho}$ for all $x \in \mathcal{X},$ \citet{tyurin2024toward} guarantee $\nicefrac{L_0 \Delta}{\varepsilon} + {\nicefrac{L_1 \Delta}{\varepsilon^{(2 - \rho) / 2}}}$ instead of $\nicefrac{(L_0 \Delta + L_1 \norm{\nabla f(x^0)}^{\rho} \Delta)}{\varepsilon}$ from \citet{li2024convex} when $0 \leq \rho \leq 2.$

\textbf{Convex optimization with $(L_0, L_1)$--smoothness and $\ell$--smoothness.} Under the $(L_0, L_1)$--smoothness assumption, convex problems were considered in \citep{koloskova2023revisiting,li2024convex,takezawa2024polyak}. \citet{gorbunov2024methods} obtained the oracle complexity $\cO\left(\nicefrac{L_0 R^2}{\varepsilon} + L_1^2 R^2\right).$ Then, the non-dominant term $L_1^2 R^2$ was improved to $\nicefrac{L_0 R^2}{\varepsilon} + \min\left\{\nicefrac{L_1 \Delta^{1 / 2} R}{\varepsilon^{1 / 2}}, L_1^2 R^2, \nicefrac{L_1 \norm{\nabla f(x^0)} R^2}{\varepsilon}\right\}$ by \citet{tyurin2024toward}. Moreover, \citet{vankov2024optimizing} improved the convergence rate of GD to $\cO\left(\nicefrac{L_0 R^2}{\varepsilon} + L_1 R \log\left(\nicefrac{\Delta}{\varepsilon}\right)\right)$ \citet{lobanov2024linear} also analyzed the possibility of improving $L_1^2 R^2$ in the region where the gradient of $f$ is large. The $\ell$--smoothness assumption in the contexts of online learning and mirror descent was considered in \citep{xie2024gradient,yu2025mirror}. 

\textbf{Accelerated convex optimization.}
The aforementioned results were derived using non-accelerated gradient descent methods. Under $(L_0, L_1)$--smoothness, accelerated variants of \algname{GD} were studied by \citet{li2024convex,gorbunov2024methods,vankov2024optimizing}. However, for small $\varepsilon,$ the approach of \citet{gorbunov2024methods} leads to the complexity $\exp(L_1 R) \nicefrac{\sqrt{L_0} R}{\sqrt{\varepsilon}}$ (up to constant factors), with an exponential dependence on $L_1$ and $R,$ while the method proposed by \citet{vankov2024optimizing} requires solving an auxiliary one-dimensional optimization problem at each iteration, leading to the oracle complexity $\cO\left(\nu \times \nicefrac{\sqrt{L_0} R}{\sqrt{\varepsilon}}\right),$ where $\nu$ is a non-constant multiplicative factor arising from solving the auxiliary problem. 
Moreover, \citet{vankov2024optimizing} analyze their method assuming that the auxiliary problem can be solved exactly, which may not be the case in practice and theory, and it is unclear how stable their method is to errors arising from solving the auxiliary problem.
In the context of the $\ell$--smoothness assumption, \citet{li2024convex} established a complexity bound of $\cO(\nicefrac{\sqrt{\ell(\norm{\nabla f(x^0)})} R}{\sqrt{\varepsilon}}).$ The current state-of-the-art accelerated methods leave open the question of whether it is possible to achieve the oracle complexities $\cO\big(\nicefrac{\sqrt{L_0} R}{\sqrt{\varepsilon}}\big)$ and $\cO\big(\nicefrac{\sqrt{\ell(0)} R}{\sqrt{\varepsilon}}\big)$ when $\varepsilon$ is small.

\emph{Remark.} We note the concurrent work \citep{borodich2025nesterov}, which also obtains optimal complexity in the small-$\varepsilon$ regime. Their result appeared online after the initial version of our paper on arXiv. Compared with \eqref{eq:iDSJQFqPsREzFtnxi} and \eqref{eq:jKlCMcGdKddup}, their bound has a larger non-dominant term and is established for the $(L_0,L_1)$--smoothness setting, whereas our result applies under the more general $\ell$--smoothness assumption. On the other hand, their algorithm is adaptive.

\subsection{Contributions}
We develop new proof techniques to analyze Algorithms~\ref{alg:main} and \ref{alg:main_new}, which, to the best of our knowledge, achieve for the first time the oracle complexities $\cO\big(\nicefrac{\sqrt{\ell(0)} R}{\sqrt{\varepsilon}} + T\big)$ and $\cO\big(\nicefrac{\sqrt{L_0} R}{\sqrt{\varepsilon}} + L_1 R \log\left(\min \left\{\nicefrac{L_1^2 \Delta}{L_0}, \nicefrac{\Delta}{\varepsilon}\right\}\right)\big)$ for all $\varepsilon > 0$ under $\ell$--smoothness and $(L_0, L_1)$--smoothness, respectively, where $T$ does not depend on $\varepsilon$.
These results represent a significant improvement over previous works \citep{li2024convex,gorbunov2024methods,vankov2024optimizing} (Table~\ref{table:complexities_convex}) since, for instance, our bound under $(L_0, L_1)$--smoothness is optimal in the small-$\varepsilon$ regime \citep{nesterov2018lectures}, and improves the previous bounds at least in the scenarios when $\varepsilon \leq L_0 / (L_1^4 R^2).$

We begin in Section~\ref{sec:subquadratic}, which establishes the $\cO\big(\nicefrac{\sqrt{\ell(0)} R}{\sqrt{\varepsilon}}\big)$ rate for small~$\varepsilon$ with subquadratic and quadratic $\ell$. In Section~\ref{sec:semistability}, we present Algorithm~\ref{alg:main_new}, which is more robust to input parameters and achieves an improved rate in the non-dominant terms, at least in the $(L_0, L_1)$--smooth case. Finally, in Section~\ref{sec:super}, we show that Algorithm~\ref{alg:main} attains the $\cO\big(\nicefrac{\sqrt{\ell(0)} R}{\sqrt{\varepsilon}}\big)$ rate (for small~$\varepsilon$) for all non-decreasing positive locally Lipschitz $\ell.$
\section{Preliminaries}
\textbf{Notations:} $\R_+ \eqdef [0, \infty);$ $\N \eqdef \{1, 2, \dots\};$ $\norm{x}$ denotes the standard Euclidean norm for all $x \in \R^d$; $\inp{x}{y} = \sum_{i=1}^{d} x_i y_i$ denotes the standard dot product; $\norm{A}$ denotes the standard spectral norm for all $A \in \R^{d \times d};$ $g = \cO(f):$ there exists $C > 0$ such that $g(z) \leq C \times f(z)$ for all $z \in \mathcal{Z};$ $g = \Omega(f):$ there exists $C > 0$ such that $g(z) \geq C \times f(z)$ for all $z \in \mathcal{Z};$ $[x]_+$ denotes $\max\{x, 0\};$
$g \simeq h:$ $g$ and $h$ are equal up to a universal positive constant; $\textnormal{Proj}_{\bar{\mathcal{X}}} (x)$ denotes the standard Euclidean projection of $x$ onto the convex closed set $\bar{\mathcal{X}}.$

We consider the following assumption \citep{li2024convex}:
\begin{assumption}
  \label{ass:gen_smooth}
  A function $f\,:\,\R^d \to \R \cup \{\infty\}$ is $\ell$--smooth if $f$ is twice differentiable on $\mathcal{X},$ $f$ is continuous on the closure of $\mathcal{X},$ and there exists a 
  \emph{non-decreasing positive locally Lipschitz} 
  function $\ell\,:\,[0, \infty) \to (0, \infty)$ such that
  \begin{align}
    \label{eq:main_ass}
    \norm{\nabla^2 f(x)} \leq \ell(\norm{\nabla f(x)})
  \end{align}
  for all $x \in \mathcal{X}.$
\end{assumption}
The assumption includes $L$--smoothness when $\ell(s) = L$, $(L_0, L_1)$--smoothness when $\ell(s) = L_0 + L_1 s$, and $(\rho, L_0, L_1)$--smoothness, i.e., $\norm{\nabla^2 f(x)} \leq L_0 + L_1 \norm{\nabla f(x)}^{\rho}$ for all $x \in \mathcal{X},$ when $\ell(s) = L_0 + L_1 s^{\rho},$ where $L, L_0, L_1, \rho \geq 0$ are some fixed constants.
While Assumption~\ref{ass:gen_smooth} requires twice differentiability, the main theorems and algorithms do not directly rely on it. Let us recall the following lemma, which follows from Assumption~\ref{ass:gen_smooth}:
\begin{lemma}[Lemma~4.3 in \cite{tyurin2024toward}]
  For all $x, y \in \mathcal{X}$ such that $\norm{y - x} \in [0, q_{\max}(\norm{\nabla f(x)})),$ if $f$ is $\ell$--smooth (Assumption~\ref{ass:gen_smooth}), then
  \begin{align}
    \label{eq:gen_smooth_alt}
    \norm{\nabla f(y) - \nabla f(x)} \leq q^{-1}(\norm{y - x};\norm{\nabla f(x)}),
  \end{align}
  where $q(s;a) \eqdef \int_{0}^{s} \frac{d v}{\ell(a + v)},$ $q^{-1}$ is the inverse of $q$ with respect to $s$, and $q_{\max}(a) \eqdef \int_{0}^{\infty} \frac{d v}{\ell(a + v)}.$
  \label{lemma:first_alt}
\end{lemma}
Not requiring twice differentiability, we can assume that \eqref{eq:gen_smooth_alt} holds instead of \eqref{eq:main_ass}. The main reason why we start with \eqref{eq:main_ass} is because it is arguably more interpretable. Next, we assume the convexity of $f$:
\begin{assumption}
  \label{ass:convex}
  A function $f\,:\,\R^d \to \R \cup \{\infty\}$ is convex and attains the minimum at a (non-unique) $x^* \in \R^d.$ We define $R \eqdef \norm{x^0 - x^*},$ where $x^0$ is a starting point of numerical methods.
\end{assumption}

In the theoretical analysis and proofs, it is useful to define the $\psi$--function:
\begin{definition}[$\psi$ and $\psi^{-1}$ functions]
  Let Assumption~\ref{ass:gen_smooth} hold. We define the function $\psi \,:\, \R_+ \to \R_+$ such that $\psi(x) = \frac{x^2}{2 \ell(4 x)},$
  and $\psi^{-1} \,:\, [0, \psi(\Delta_{\max})) \to [0, \Delta_{\max})$ as its (standard) inverse, where $\Delta_{\max} \in (0, \infty]$ is the largest constant such that $\psi$ is strictly increasing on\footnote{$\Delta_{\max} > 0$ due to Lemma~\ref{lemma:func_to_grad_spec_diff}.} $[0, \Delta_{\max})$.
  \label{def:psi_function}
\end{definition}

\section{Subquadratic and Quadratic Growth of $\ell$}
\begin{algorithm}[t]
  \caption{Accelerated Gradient Descent (\algname{AGD}) with $\ell$-Smoothness}
  \label{alg:main}
  \begin{algorithmic}[1]
  \STATE \textbf{Input:} starting point $x^0 \in \mathcal{X},$ function $\ell$ from Assumption~\ref{ass:gen_smooth}, parameters $\delta$ and $\bar{R}$
  \STATE Starting from $x^0,$ run \algname{GD} from \citep[Alg.~1]{tyurin2024toward} until $f(\bar{x}) - f(x^*) \leq \delta / 2,$ \\
  where $\bar{x}$ is the output point of \algname{GD}
  \STATE Init $y^0 = u^0 = \bar{x}$
  \STATE Set $\Gamma_0 = \delta / \bar{R}^2$
  \STATE Set $\gamma = 1 / \left(2 \ell\left(0\right)\right)$
  \FOR{$k = 0, 1, \dots$}
  \STATE $\alpha_k = \sqrt{\gamma \Gamma_k}$
  \STATE $y^{k+1} = \frac{1}{1 + \alpha_k} y^{k} + \frac{\alpha_k}{1 + \alpha_k} u^{k} - \frac{\gamma}{1 + \alpha_k} \nabla f(y^{k})$
  \STATE $u^{k+1} = \textnormal{Proj}_{\bar{\mathcal{X}}} \left(u^k - \frac{\alpha_k}{\Gamma_k} \nabla f(y^{k+1})\right)$ \\ ($\bar{\mathcal{X}}$ is the closure of $\mathcal{X}$)
  \STATE $\Gamma_{k+1} = \Gamma_{k} / (1 + \alpha_k)$
  \ENDFOR
  \end{algorithmic}
\end{algorithm}
\label{sec:subquadratic}
We are ready to present our first result. Consider Algorithm~\ref{alg:main}, which consists of two phases: first, we run (non-accelerated) \algname{GD}, and then we run an accelerated version of \algname{GD}. Later, we will present Algorithm~\ref{alg:main_new}, which avoids the first phase. We first state the convergence rate of Algorithm~\ref{alg:main} and then discuss and explain it in more detail. We begin by stating a standard result from the theory of accelerated methods \citep{nesterov2018lectures,lan2020first,stonyakin2021inexact} concerning auxiliary sequences, which control convergence rates:
\begin{restatable}{theorem}{THEOREMSEQ}
  \label{thm:gamma}
  For any $\Gamma_0 > 0$ and $\gamma \geq 0,$ let $\alpha_k \geq \sqrt{\gamma \Gamma_k}$ and $\Gamma_{k+1} = \Gamma_{k} / (1 + \alpha_k)$ for all $k \geq 0.$ Then, $\textstyle \Gamma_{k+1} \leq \frac{9}{\gamma \left(k + 1 - \bar{k}\right)^2}$
  for all $k \geq \bar{k} \eqdef \max\left\{1 + \frac{1}{2} \log_{3/2}\left(\frac{\gamma \Gamma_0}{4}\right), 0\right\}.$
\end{restatable}
The following result provides the convergence rate of Algorithm~\ref{alg:main} for $\ell$ such that $\psi(x) = \frac{x^2}{2 \ell(4 x)}$ is strictly increasing, which holds, for instance, under $(L_0, L_1)$--smoothness. Note that in Section~\ref{sec:super}, we extend the following result to not necessarily increasing $\psi$.
\begin{restatable}{theorem}{THEORESUB}
  \label{thm:convex_increasing}
  Suppose that Assumptions~\ref{ass:gen_smooth} and \ref{ass:convex} hold. Let $\psi \,:\, \R_{+} \to \R_{+}$ such that $\psi(x) = \frac{x^2}{2 \ell(4 x)}$ be strictly increasing. Then Algorithm~\ref{alg:main}
  guarantees that
  \begin{align}
    \label{eq:lGLpMDrQW}
    \textstyle f(y^{k+1}) - f(x^*) \leq \Gamma_{k+1} \bar{R}^2 \leq \frac{18 \ell\left(0\right) \bar{R}^2}{\left(k + 1 - \bar{k}\right)^2}
  \end{align}
  for all $k \geq \bar{k} \eqdef \max\left\{1 + \frac{1}{2} \log_{3/2}\left(\frac{\Gamma_0}{8 \ell\left(0\right)}\right), 0\right\}$ with any $\delta \in (0, \infty]$ such that $\ell\left(8 \sqrt{\delta \ell\left(0\right)}\right) \leq 2 \ell\left(0\right)$ and any $\bar{R} \geq R \eqdef \norm{x^{0} - x^*}.$ 
\end{restatable}
The theorem establishes the desired $\nicefrac{\ell\left(0\right) R^2}{k^2}$ convergence rate of accelerated methods. However, the method enters this regime only after running the \algname{GD} method and after the initial $\bar{k}$ steps of the accelerated steps. The main and final result in this section, which captures the total oracle complexity, is presented below.
\begin{restatable}{theorem}{THEORESUBRATE}
  \label{thm:convex_increasing_rate}
  Consider the assumptions and results of Theorem~\ref{thm:convex_increasing}. The oracle complexity (i.e., the number of gradient calls) required to find an $\varepsilon$--solution is
\begin{align}
  \label{eq:LHFPooCjyEV}
  \frac{5 \sqrt{\ell(0)} \bar{R}}{\sqrt{\varepsilon}} + k(\delta),
\end{align}
for all $\delta \geq 0$ such that $\ell\left(8 \sqrt{\delta \ell\left(0\right)}\right) \leq 2 \ell\left(0\right),$
where 
  $k(\delta) \eqdef \max\left\{1 + \frac{1}{2} \log_{3/2}\left(\frac{\delta}{8 \ell(0) \bar{R}^2}\right),\, 0\right\} + k_{\textnormal{\algname{GD}}}(\delta),$
$k_{\textnormal{\algname{GD}}}(\delta)$ is the oracle complexity of \algname{GD} for finding a point $\bar{x}$ such that $f(\bar{x}) - f(x^*) \leq \delta / 2.$ 
\end{restatable}
\begin{corollary}
  \label{cor:first}
  In Theorem~\ref{thm:convex_increasing_rate}, minimizing over $\delta$ and taking $\bar{R} = R \eqdef \norm{x^{0} - x^*},$ the oracle complexity is 
  \begin{align}
    \label{eq:GniMjHfKmHVLXLzy}
    \frac{5 \sqrt{\ell(0)} R}{\sqrt{\varepsilon}} + \underbrace{\min_{\delta \geq 0\,:\,\ell\left(8 \sqrt{\delta \ell\left(0\right)}\right) \leq 2 \ell\left(0\right)}k(\delta)}_{\textnormal{does not depend on $\varepsilon$}}.
  \end{align}
\end{corollary}
\subsection{Example: $(L_0, L_1)$--smoothness}
\label{sec:example_1}
We now consider an example and apply the result for $(L_0, L_1)$--smooth functions. In this case, $\ell(s) = L_0 + L_1 s.$ First, we need to find the proper set of $\delta$ from Theorem~\ref{thm:convex_increasing}: $\ell(8 \sqrt{\delta \ell\left(0\right)}) \leq 2 \ell(0) \Leftrightarrow L_0 + L_1 (8 \sqrt{\delta L_0}) \leq 2 L_0 \Leftrightarrow  \delta \leq L_0 / (64 L_1^2).$ Second, we need to find $k_{\textnormal{\algname{GD}}}(\delta).$ Using Table~2 from \citep{tyurin2024toward}, or the results by \citet{gorbunov2024methods,vankov2024optimizing}, $k_{\textnormal{\algname{GD}}}(\delta) = \cO\left(\nicefrac{L_0 R^2}{\delta} + \min\left\{\nicefrac{L_1 \Delta^{1 / 2} R}{\delta^{1 / 2}}, L_1^2 R^2, \nicefrac{L_1 \norm{\nabla f(x^0)} R^2}{\delta}\right\}\right) = \cO\left(\frac{L_0 R^2}{\delta}\right) = \cO\left(\frac{L_0 \bar{R}^2}{\delta}\right)$ for all $\delta \leq L_0 / (64 L_1^2).$ Substituting to \eqref{eq:LHFPooCjyEV}, we get the total oracle complexity
\begin{align}
  \label{eq:mkMctLlqZ}
  \textstyle \cO\left(\frac{\sqrt{L_0} \bar{R}}{\sqrt{\varepsilon}} + \min\limits_{0 \leq \delta \leq L_0 / (64 L_1^2)} \left[\left[\log \left(\frac{\delta}{L_0 \bar{R}^2}\right)\right]_+ + \frac{L_0 \bar{R}^2}{\delta}\right]\right),
\end{align}
where $[x]_+ \equiv \max\{x, 0\}$ for all $x \in \R.$ Taking $\delta = \min\{L_0 / (64 L_1^2), (L_0 \bar{R}^2) / 64\}$ (which might not be the optimal choice, but a sufficient choice to show that the first term dominates if $\varepsilon$ is small), we get
\begin{align}
  \label{eq:iDSJQFqPsREzFtnxi}
  \textstyle \eqref{eq:mkMctLlqZ} = \cO\left(\frac{\sqrt{L_0} \bar{R}}{\sqrt{\varepsilon}} + L_1^2 \bar{R}^2\right) = \cO\left(\frac{\sqrt{L_0} R}{\sqrt{\varepsilon}} + L_1^2 R^2\right),
\end{align} 
where we choose $\bar{R} = R.$ Unlike \citet{li2024convex,gorbunov2024methods,vankov2024optimizing}, we get $\cO(\nicefrac{\sqrt{L_0} R}{\sqrt{\varepsilon}})$ for small $\varepsilon.$ Moreover, at least in the regime $\varepsilon \leq L_0 / (L_1^4 R^2),$ this complexity is optimal \citep{nemirovskij1983problem,nesterov2018lectures} since for any $L_0 > 0,$ and $L_1 \geq 0,$ it is possible to find an $(L_0, L_1)$--smooth function (the $(L_0, 0)$--smooth function from Section~2.1.2 of \citep{nesterov2018lectures}) such that the required number of oracle calls is $\Omega(\nicefrac{\sqrt{L_0} R}{\sqrt{\varepsilon}}).$ In Section~\ref{sec:spec}, we provide a better dependence on the non-dominant term.

One can repeat these steps for any $\ell$ such that $\psi$ is strictly increasing. Nevertheless, even without these derivations, we establish the total oracle complexity $\cO(\nicefrac{\sqrt{\ell(0)} R}{\sqrt{\varepsilon}})$ in \eqref{eq:GniMjHfKmHVLXLzy} for small $\varepsilon$.

\subsection{Discussion}
The closest work to the complexity $\cO\left(\nicefrac{\sqrt{L_0} R}{\sqrt{\varepsilon}}\right),$ when $\varepsilon$ is small, is \citep{vankov2024optimizing}. Using the same idea as in \citep{vankov2024optimizing}, in Algorithm~\ref{alg:main}, we run \algname{GD} until $f(\bar{x}) - f(x^*) \leq \nicefrac{\delta}{2}$. However, the next steps and proof techniques are new. Using the ``warm-start'' point $\bar{x},$ it becomes easier for Algorithm~\ref{alg:main} to run accelerated steps because we take $\delta$ such that $\ell(4 \norm{\nabla f(y^0)}) \leq 2 \ell(0)$ (Lemma~\ref{lemma:bound_sub}), meaning that we start from the region where the local smoothness constant is almost $\ell(0).$ The main challenge is to ensure that the next points $y^k$ of Algorithm~\ref{alg:main} never leave this region. To ensure that, using the method from \citep{nesterov2021primal}, \citet{vankov2024optimizing} utilize the monotonicity of their accelerated method and the fact that their points do not leave the region with small smoothness. However, this improvement is not free and requires solving an auxiliary problem and $\nu$ extra oracle calls at each iteration, where $\nu$ is not a universal constant and depends on the parameters of $f$, leading to a suboptimal complexity.

In contrast, our method follows the standard approach, where only one gradient is computed per iteration. We use the version of the accelerated method from \citep{weiaccelerated}[Section D.2], with some minor but important modifications. The method itself is very similar to the one from \citep{allen2014linear}, for instance. However, the proof technique is very different, which is the main reason we focus on Algorithm~\ref{alg:main}. While for $L$--smooth functions, up to constant factors, the proof technique from \citep{weiaccelerated} does not offer any advantages over \citep{nesterov1983method} because the result in \citep{nesterov1983method} is optimal. In the case of functions with generalized smoothness, it becomes particularly useful, as shown in the following section.
\subsection{Proof sketch}
\label{sec:proof_sketch}
As in most proofs, we define the Lyapunov function 
$V_{k} \eqdef f(y^{k}) - f(x^*) + \frac{\Gamma_{k}}{2} \norm{u^{k} - x^*}^2.$ 
The first important observation is that in $V_k$ we use $y^k$, the point where the gradient is actually computed. 
This is important, and we will see why later.

Using mathematical induction, let us assume that we have run Algorithm~\ref{alg:main} up to $k$\textsuperscript{th} iteration, $\ell\left(4 \norm{\nabla f(y^k)}\right) \leq 2 \ell(0),$ and $V_{k} \leq \left(\prod_{i=0}^{k-1} \frac{1}{1 + \alpha_i}\right) V_0.$ We choose $\Gamma_0$ such that $V_0 \leq \delta.$ The base case with $k = 0$ is true because we run \algname{GD} until $\ell(4 \norm{\nabla f(y^0)}) \leq 2 \ell(0).$ Now, instead of $k + 1$\textsuperscript{th} consider the steps 
\begin{equation}
  \label{eq:hzUSJDK}
\begin{aligned}
  \alpha_{k,\gamma} &= \sqrt{\gamma \Gamma_k}, \\
  y^{k+1}_{\gamma} &= \frac{1}{1 + \alpha_{k,\gamma}} y^{k} + \frac{\alpha_{k,\gamma}}{1 + \alpha_{k,\gamma}} u^{k} - \frac{\gamma}{1 + \alpha_{k,\gamma}} \nabla f(y^{k}), \\
  u^{k+1}_{\gamma} &= \textnormal{Proj}_{\bar{\mathcal{X}}} \left(u^k - \frac{\alpha_{k,\gamma}}{\Gamma_k} \nabla f(y^{k+1}_{\gamma})\right), \\
  \Gamma_{k+1, \gamma} &= \Gamma_{k} / (1 + \alpha_{k,\gamma}),
\end{aligned}
\end{equation}
where $\gamma$ is a free parameter. These steps are equivalent to $k + 1$\textsuperscript{th} iteration when $\gamma = 1 / \left(2 \ell\left(0\right)\right).$ However, we have not proved that we are allowed to use this $\gamma$ yet. For these steps, we can prove a standard descent lemma, Lemma~\ref{thm:main_lemma}:
\begin{align}
  \label{eq:hAnZdfCzZSdqlHD}
  &\textstyle (1 + \alpha_{k,\gamma}) V_{k+1,\gamma} \leq V_k \nonumber \\
  &\textstyle \quad + \frac{1}{2} \Big[\left(\gamma - \frac{1}{\ell(2 \|\nabla f(y^{k})\| + \|\nabla f(y^{k + 1}_{\gamma})\|)}\right) \nonumber \\
  &\textstyle \quad \qquad \times \norm{\nabla f(y^{k+1}_{\gamma}) - \nabla f(y^{k})}^2\Big],
\end{align}
where
\begin{align*}
  &\textstyle V_{k+1,\gamma} \eqdef f(y^{k+1}_{\gamma}) - f(x^*) + \frac{\Gamma_{k+1, \gamma}}{2} \norm{u^{k+1}_{\gamma} - x^*}^2.
\end{align*}
For now, let us assume that $f$ is $L$--smooth. Then the rest of the proof becomes straightforward. In this case, $\ell(2 \norm{\nabla f(y^{k})} + \norm{\nabla f(y^{k + 1}_{\gamma})}) = L,$ and we can take $\gamma = 1 / 2 L \equiv 1 / (2 \ell(0))$ to ensure that 
$(1 + \alpha_k) V_{k+1} \leq V_k$ because $(1 + \alpha_{k,\gamma}) V_{k+1,\gamma} = (1 + \alpha_k) V_{k+1}$ with $\gamma = 1 / 2 L.$ Then, we should unroll the recursion and use Theorem~\ref{thm:gamma} to get the classical $1 / k^2$ rate \citep{nesterov1983method}.

However, under Assumption~\ref{ass:gen_smooth}, $\ell(2 \norm{\nabla f(y^{k})} + \norm{\nabla f(y^{k + 1}_{\gamma})})$ depends on $\norm{\nabla f(y^{k + 1}_{\gamma})}$, and we encounter a ``chicken-and-egg'' dilemma: in order to choose $\gamma$, we need to know $\norm{\nabla f(y^{k + 1}_{\gamma})}$, which in turn depends on $\gamma.$ Our resolution is the following. Let us take (non-explicitly) the smallest $\gamma^* \geq 0$ such that 
\begin{align*}
  g(\gamma) \eqdef \gamma - \frac{1}{\ell(2 \norm{\nabla f(y^{k})} + \norm{\nabla f(y^{k + 1}_{\gamma})})} = 0,
\end{align*}
which exists and is positive because $g(\gamma)$ is continuous, $g(0) < 0,$ and $g(\bar{\gamma}) \geq 0$ for $\bar{\gamma} = \frac{1}{\ell(2 \norm{\nabla f(y^{k})})}.$ It is possible that we are ``unlucky'' and $\gamma^*$ is very small, leading to a slow convergence rate and preventing us from choosing $\gamma = 1 / (2 \ell(0)).$ Surprisingly, it is possible to show that $\gamma^* \geq 1 / (2 \ell(0)).$ Indeed, using \eqref{eq:hAnZdfCzZSdqlHD}, for all $\gamma \leq \gamma^*,$ we have $f(y^{k+1}_{\gamma}) - f(x^*) \leq V_k \leq V_0.$ Recall that we choose $\Gamma_0$ such that $V_0 \leq \delta.$ Thus, $f(y^{k+1}_{\gamma}) - f(x^*) \leq \delta.$ This is the key inequality in the proof, which allows us to conclude that the function gap with $y^{k+1}_{\gamma}$ is bounded, thus justifying the choice of the Lyapunov function.

It left to use Lemma~\ref{lemma:bound_sub}, which allows us to bound $\ell(4 \norm{\nabla f(y)})$ if we can bound $f(y) - f(x^*) \leq \delta$ for all $y \in \mathcal{X}.$ Thus, $\ell\left(4 \norm{\nabla f(y^{k+1}_{\gamma})}\right) \leq 2 \ell(0)$ for all $\gamma \leq \gamma^*.$ Recalling the definition of $\gamma^*:$
\begin{align*}
  \textstyle \gamma^* &= \textstyle  \frac{1}{\ell(2 \norm{\nabla f(y^{k})} + \norm{\nabla f(y^{k + 1}_{\gamma^*})})} \\
  & \textstyle \geq \frac{1}{\max\{\ell(4 \norm{\nabla f(y^{k})}), \ell(4 \norm{\nabla f(y^{k + 1}_{\gamma^*})})\}} \geq \frac{1}{2 \ell(0)}.
\end{align*}
Finally, this means that we can take $\gamma = 1 / (2 \ell(0)),$ \eqref{eq:hzUSJDK} reduces to the $k + 1$\textsuperscript{th} step of Algorithm~\ref{alg:main}, $\ell\left(4 \norm{\nabla f(y^{k+1})}\right) \leq 2 \ell(0),$ and $V_{k + 1} \leq \left(\prod_{i=0}^{k} \frac{1}{1 + \alpha_i}\right) V_0$ due to \eqref{eq:hAnZdfCzZSdqlHD}. We have proved the next step of mathematical induction and \eqref{eq:lGLpMDrQW}.

The way we resolve the ``chicken-and-egg'' dilemma can be an interesting proof trick in other optimization contexts. Note that our method is not necessarily monotonic, but the proof still allows us to show that the method never leaves the region where the local smoothness constant is almost $\ell(0).$

\newcommand{\sectionname}{Stability with Respect to Input Parameters and Improved Rates}
\section{\sectionname}
\label{sec:semistability}
While, to the best of our knowledge, Algorithm~\ref{alg:main} is the first algorithm with $\cO\big(\nicefrac{\sqrt{\ell(0)} R}{\sqrt{\varepsilon}}\big)$ complexity, it has two limitations: it runs \algname{GD} at the beginning, and it requires a good estimate of $R$ when selecting $\bar{R}.$ We resolve these issues in Algorithm~\ref{alg:main_new}, which is similar to Algorithm~\ref{alg:main}, but the former does not run \algname{GD} at the beginning, uses the step sizes $\gamma_k = 1 / \ell\left(4 \psi^{-1}\left(\Gamma_k \bar{R}^2 \right)\right),$ and requires $\Gamma_0$ as an input.
\begin{restatable}{theorem}{THEORESUBBETTER}
  \label{thm:convex_increasing_better}
  Suppose that Assumptions~\ref{ass:gen_smooth} and \ref{ass:convex} hold. Let $\psi \,:\, \R_{+} \to \R_{+}$ such that $\psi(x) = \frac{x^2}{2 \ell(4 x)}$ be strictly increasing and $\lim\limits_{x \to \infty} \psi(x) = \infty.$ Then Algorithm~\ref{alg:main_new}
  guarantees that
  \begin{align*}
    f(y^{k+1}) - f(x^*) \leq \Gamma_{k+1} R^2
  \end{align*}
  for all $k \geq 0$ with $\Gamma_{0} \geq \frac{2 (f(x^{0}) - f(x^*))}{\norm{x^{0} - x^*}^2}$ and $\bar{R} \geq R.$
\end{restatable}
\begin{algorithm}[t]
  \caption{\algname{AGD} with $\ell$-smoothness and increasing step sizes (without \algname{GD} pre-running)}
  \label{alg:main_new}
  \begin{algorithmic}[1]
  \STATE \textbf{Input:} starting point $x^0 \in \mathcal{X},$ function $\ell$ from Assumption~\ref{ass:gen_smooth}, parameters $\Gamma_0$ and $\bar{R}$
  \STATE Init $y^0 = u^0 = x^0$
  \STATE Define $\psi(x) = \frac{x^2}{2 \ell(4 x)}$ \hfill (assume that $\psi$ is invertible on $\R_+$; Algorithm~\ref{alg:main} does not require this)
  \FOR{$k = 0, 1, \dots$}
  \STATE $\gamma_k = 1 / \ell\left(4 \psi^{-1}\left(\Gamma_k \bar{R}^2 \right)\right)$
  \STATE $\alpha_k = \sqrt{\gamma_k \Gamma_k}$
  \STATE $y^{k+1} = \frac{1}{1 + \alpha_k} y^{k} + \frac{\alpha_k}{1 + \alpha_k} u^{k} - \frac{\gamma_k}{1 + \alpha_k} \nabla f(y^{k})$
  \STATE $u^{k+1} = \textnormal{Proj}_{\bar{\mathcal{X}}} \left(u^k - \frac{\alpha_k}{\Gamma_k} \nabla f(y^{k+1})\right)$ \\ ($\bar{\mathcal{X}}$ is the closure of $\mathcal{X}$)
  \STATE $\Gamma_{k+1} = \Gamma_{k} / (1 + \alpha_k)$
  \ENDFOR
  \end{algorithmic}
\end{algorithm}
\begin{restatable}{theorem}{THEORESUBBETTERRATE}
  \label{thm:convex_increasing_better_rate}
  Consider the assumptions and results of Theorem~\ref{thm:convex_increasing_better}. The oracle complexity (i.e., the number of gradient calls) required to find an $\varepsilon$--solution is
\begin{align}
  \label{eq:JXpxPDWSQrnCDAPLJgH}
  \frac{5 \sqrt{\ell(0)} R}{\sqrt{\varepsilon}} + \underbrace{\max\left\{2 + \log_{3/2}\left(\frac{\Gamma_{0}}{4 \ell(0)}\right), 0\right\} + k_{\textnormal{init}}}_{\textnormal{does not depend on $\varepsilon$}}
\end{align}
with $\Gamma_0 \geq \frac{2 (f(x^0) - f(x^*))}{\norm{x^0 - x^*}^2},$ $\bar{R} \geq R,$ and $k_{\textnormal{init}}$ being the smallest integer such that 
  $$\textstyle \ell\left(24 \sqrt{\frac{\ell\left(4 \psi^{-1}\left(\Gamma_0 \bar{R}^2 \right)\right) \ell\left(0\right) \bar{R}^2}{k_{\textnormal{init}}^2} }\right) \leq 2 \ell\left(0\right).$$
\end{restatable}
Comparing \eqref{eq:JXpxPDWSQrnCDAPLJgH} and \eqref{eq:mkMctLlqZ}, one can see that the first term in \eqref{eq:JXpxPDWSQrnCDAPLJgH} does not depend on to the choice of $\bar{R}$ and $\Gamma_0.$ Ideally, it is better to choose $\Gamma_{0} = \frac{2 (f(x^{0}) - f(x^*))}{\norm{x^{0} - x^*}^2}$ and $\bar{R} = R.$ However, if we overestimate $\bar{R}$ and $\Gamma_0,$ the penalty for this appears in the term that does not depend on $\varepsilon.$ In the next section, we consider an example to illustrate this.
\subsection{Example: $(L_0, L_1)$--smoothness}
\label{sec:example_2}
To find the oracle complexity, we have to estimate $k_{\textnormal{init}}.$ In the case of $(L_0, L_1)$--smoothness, we can find $k_{\textnormal{init}}$ from the equality 
$L_0 + L_1 \sqrt{(L_0 + L_1 \psi^{-1}\left(\Gamma_0 \bar{R}^2 \right)) L_0 \bar{R}^2 / k_{\textnormal{init}}^2 } \simeq 2 L_0$ 
(we ignore constants for simplicity), where $\psi^{-1}$ is the inverse of $x^2 / (2 (L_0 + 4 L_1 x)).$ If $\Gamma_0 \bar{R}^2 \geq L_0 / L_1,$ then the equality is equivalent to $k_{\textnormal{init}} \simeq \sqrt{L_1^2 \bar{R}^2 + L_1^4 \Gamma_0 \bar{R}^4 / L_0}.$ Otherwise, $k_{\textnormal{init}} \simeq \sqrt{L_1^2 \bar{R}^2 + L_1^3 \bar{R}^3 \sqrt{\Gamma_0 / L_0}}.$ Thus, using \eqref{eq:JXpxPDWSQrnCDAPLJgH}, the total oracle complexity is
\begin{align}
  \label{eq:UPABGqcdwwz}
  \textstyle \cO\left(\frac{\sqrt{L_0} R}{\sqrt{\varepsilon}} + L_1 \bar{R} + L_1^2 \bar{R}^2 \sqrt{\frac{\Gamma_0}{L_0}} + \left[\log\left(\frac{\Gamma_{0}}{L_0}\right)\right]_+\right),
\end{align}
where the first term is stable to the choice of $\bar{R}$ and $\Gamma_0.$

\subsection{Specialization for $(L_0, L_1)$--smoothness}
\label{sec:spec}
The previous theorems work with any $\ell$ such that $\psi(x) = \frac{x^2}{2 \ell(4 x)}$ is strictly increasing on $\R_+$ and $\lim\limits_{x \to \infty} \psi(x) = \infty.$ It turns out that we can improve \eqref{eq:UPABGqcdwwz} and refine Theorem~\ref{thm:convex_increasing_better_rate} in the case of $(L_0, L_1)$--smoothness.
\begin{restatable}{theorem}{THEORESUBBETTERRATEREFINE}
  \label{thm:convex_increasing_better_rate_refine}
  Consider the assumptions and results of Theorem~\ref{thm:convex_increasing_better} with $\ell(s) = L_0 + L_1 s.$ The oracle complexity (i.e., the number of gradient calls) required to find an $\varepsilon$--solution is
\begin{equation}
\label{eq:jKlCMcGdKddup}
\begin{aligned}
  \textstyle \cO\Big(& \textstyle \frac{\sqrt{L_0} R}{\sqrt{\varepsilon}} + \max\left\{L_1 \bar{R}\log\left(\min \left\{\frac{L_1^2 \bar{R}^2 \Gamma_0}{L_0}, \frac{\Gamma_0 R^2}{\varepsilon}\right\}\right), 0\right\} \\
  &\textstyle \qquad + \max\left\{\log\left(\frac{\Gamma_{0}}{L_0}\right), 0\right\}\Big)
\end{aligned}
\end{equation}
with $\Gamma_0 \geq \frac{2 (f(x^0) - f(x^*))}{\norm{x^0 - x^*}^2}$ and $\bar{R} \geq R.$
\end{restatable}
The non-dominant term in \eqref{eq:jKlCMcGdKddup} is better than that of \eqref{eq:UPABGqcdwwz}, and is better than that of \eqref{eq:iDSJQFqPsREzFtnxi} when $\Gamma_0 = \nicefrac{2 \Delta}{R^2}$ and $\bar{R} = R.$

\subsection{Discussion and proof sketch}
Unlike Algorithm~\ref{alg:main}, Algorithm~\ref{alg:main_new} starts from $x^0$ where the initial local smoothness might be large. Nevertheless, the proof follows the proof techniques from Section~\ref{sec:proof_sketch} with one important difference: using mathematical induction, we prove that $\norm{\nabla f(y^{k})} \leq \psi^{-1}(\Gamma_{k} \bar{R}^2)$ for all $k \geq 0.$ This inequality means that $\norm{\nabla f(y^{k})}$ can be bounded by a decreasing sequence, and after several iterations, all $y^k$ satisfy $\ell(4 \norm{\nabla f(y^{k})}) \leq 2 \ell(0),$ allowing us to get $\cO(\nicefrac{\sqrt{\ell(0)} R}{\sqrt{\varepsilon}})$ complexity for small--$\varepsilon.$

\subsection{Comparison of the non-dominant term with previous methods}
For $(L_0,L_1)$--smoothness, according to Table~\ref{table:complexities_convex}, our non-dominant term is at most $L_1 R \log(\min \{\nicefrac{L_1^2 \Delta}{L_0}, \nicefrac{\Delta}{\varepsilon}\}) \leq L_1^2 R^2$, which is much better than the complexity of \citep{gorbunov2024methods}, since the latter has an exponential dependence on $L_1 R$. At the same time, the non-dominant term $\nu \times (L_1 R)^{2/3} \log(\nicefrac{\Delta}{\varepsilon})$ of \citep{vankov2024optimizing} can be better than ours, but our dominant term is better, so here we do not offer a uniform improvement. Finally, the non-dominant term in \citep{li2024convex} is worse than ours, since their term is at least $L_1^2 R^2 \sqrt{(\Delta + L_0 R^2) / \varepsilon}$, as shown in \citep[Section~F]{vankov2024optimizing}. Overall, across all regimes, our complexity is better than those in \citep{li2024convex,gorbunov2024methods}. We do not offer a uniform improvement over \citep{vankov2024optimizing}, but we are better when $\varepsilon$ is small.

\section{Superquadratic Growth of $\ell$}
\label{sec:super}
In the previous sections, we provided convergence rates under the assumption that $\psi$ is strictly increasing. For instance, the previous theory applies to $(\rho, L_0, L_1)$--smooth functions only if $\rho \leq 2.$ For cases where $\psi$ is not necessarily strictly increasing, we can prove the following theorems.
\begin{restatable}{theorem}{THEORESUPER}
  \label{thm:convex_increasing_super}
  Suppose that Assumptions~\ref{ass:gen_smooth} and \ref{ass:convex} hold. Let $\psi \,:\, \R_{+} \to \R_{+}$ such that $\psi(x) = \frac{x^2}{2 \ell(4 x)}$ be not necessarily strictly increasing. Find the largest $\Delta_{\max} \in (0, \infty]$ such that $\psi$ is strictly increasing on $[0, \Delta_{\max}).$ 
  For all $\delta \in [0, \psi(\Delta_{\max})),$ find the unique $\Delta_{\textnormal{left}}(\delta) \in [0, \Delta_{\max})$ and the smallest\footnote{if the set $\{x \in [\Delta_{\max}, \infty) \,:\, \psi(x) = \delta\}$ is empty, then $\Delta_{\textnormal{right}}(\delta) = \infty$} $\Delta_{\textnormal{right}}(\delta) \in [\Delta_{\max}, \infty]$ such that $\psi(\Delta_{\textnormal{left}}(\delta)) = \delta$ and $\psi(\Delta_{\textnormal{right}}(\delta)) = \delta.$ Take any $\delta \in [0, \frac{1}{2} \psi(\Delta_{\max})]$ such that $\ell(4 \Delta_{\textnormal{left}}(\delta)) \leq 2 \ell(0),$ $\Delta_{\textnormal{right}}(\delta) \geq 2 M_{\bar{R}},$ and $\delta \leq \Delta,$ where\footnote{or it is sufficient to find any $M_{\bar{R}}$ such that $M_{\bar{R}}~\geq~ \max_{f(x) - f(x^*) \leq \Delta, \norm{x - x^*} \leq 2 \bar{R}} \norm{\nabla f(x)}.$ The latter term is well-defined since $\{x \in \R^d \, | \, f(x) - f(x^*) \leq \Delta, \norm{x - x^*} \leq 2 \bar{R}\}$ is compact and strictly belongs to $\mathcal{X}.$} $M_{\bar{R}} \eqdef \max\limits_{f(x) - f(x^*) \leq \Delta, \norm{x - x^*} \leq 2 \bar{R}} \norm{\nabla f(x)}.$
  Then Algorithm~\ref{alg:main}
  guarantees that
  \begin{align*}
    f(y^{k+1}) - f(x^*) \leq \Gamma_{k+1} \bar{R}^2 \leq \frac{18 \ell\left(0\right) \bar{R}^2}{\left(k + 1 - \bar{k}\right)^2}
  \end{align*}
  for all $k \geq \bar{k} \eqdef \max\left\{1 + \frac{1}{2} \log_{3/2}\left(\frac{\Gamma_0}{8 \ell\left(0\right)}\right), 0\right\}$ with any $\bar{R} \geq \norm{x^{0} - x^*}.$
\end{restatable}
In order to apply the theorem and algorithm, we first have to find the largest $\Delta_{\max} \in (0, \infty]$ such that $\psi$ is strictly increasing on $[0, \Delta_{\max}).$ If $\psi$ is strictly increasing on $\R_+,$ then $\Delta_{\max} = \infty.$ Next, we should find $\Delta_{\textnormal{left}}(\delta)$ and $\Delta_{\textnormal{right}}(\delta)$ for all $\delta \in [0, \psi(\Delta_{\max})).$ The point $\Delta_{\textnormal{left}}(\delta) \in [0, \Delta_{\max})$ is the solution of $\psi(\Delta_{\textnormal{left}}(\delta)) = \delta,$ which exists and is unique for all $\delta \in [0, \psi(\Delta_{\max}))$ because $\psi$ is strictly increasing on $[0, \Delta_{\max}).$ Notice that $\psi(x) > \delta$ for all $x \in (\Delta_{\textnormal{left}}(\delta), \Delta_{\max}).$ Thus, there are two options: either $\psi(x) > \delta$ for all $x \in (\Delta_{\textnormal{left}}(\delta), \infty),$ and we define $\Delta_{\textnormal{right}}(\delta) = \infty,$ or there exists the first moment $\Delta_{\textnormal{right}}(\delta) \in [\Delta_{\max}, \infty)$ when $\psi(\Delta_{\textnormal{right}}(\delta)) = \delta.$ In other words, $\Delta_{\textnormal{right}}(\delta)$ is the second time when $\psi$ intersects $\delta.$ We define the set of $\delta$ allowed to use in the algorithm:
\begin{align*}
  \textstyle Q \eqdef 
  \Big\{&\delta \in [0, \psi(\Delta_{\max}) / 2]:\ell(4 \Delta_{\textnormal{left}}(\delta)) \leq 2 \ell(0), \\
  &\qquad \Delta_{\textnormal{right}}(\delta) \geq 2 M_{\bar{R}}, \delta \leq \Delta\Big\},
\end{align*}
which is non-empty due to Lemma~\ref{lemma:func_to_grad_spec_diff}.
\begin{restatable}{theorem}{THEORESUPERRATE}
  \label{thm:convex_increasing_super_rate}
  Consider the assumptions and results of Theorem~\ref{thm:convex_increasing_super}. The oracle complexity (i.e., the number of gradient calls) required to find an $\varepsilon$--solution is
\begin{align*}
  \frac{5 \sqrt{\ell(0)} \bar{R}}{\sqrt{\varepsilon}} + k(\delta)
\end{align*}
for all $\delta \in Q,$ where 
  $k(\delta) \eqdef \max\left\{1 + \frac{1}{2} \log_{3/2}\left(\frac{\delta}{8 \ell(0) \bar{R}^2}\right),\, 0\right\} + k_{\textnormal{\algname{GD}}}(\delta),$
$k_{\textnormal{\algname{GD}}}(\delta)$ is the oracle complexity of \algname{GD} for finding a point $\bar{x}$ such that $f(\bar{x}) - f(x^*) \leq \delta / 2.$ 
\end{restatable}
\begin{corollary}
  \label{cor:general}
  In Theorem~\ref{thm:convex_increasing_super_rate}, minimizing over $\delta$ and taking $\bar{R} = R \eqdef \norm{x^{0} - x^*},$ the oracle complexity is
  \begin{align}
    \label{eq:bXpbllgkiSQgRkTr}
    \frac{5 \sqrt{\ell(0)} R}{\sqrt{\varepsilon}} + \underbrace{\min_{\delta \in Q} k(\delta)}_{\textnormal{does not depend on $\varepsilon$}}.
  \end{align}
\end{corollary}
In Section~\ref{sec:example_smooth}, we consider an example, $(\rho, L_0, L_1)$--smoothness, to illustrate how to use the theorem, and show that it guarantees a rate of $\nicefrac{\sqrt{L_0} R}{\sqrt{\varepsilon}}$ rate for any $\rho \geq 0$ and a sufficiently small $\varepsilon.$ The main observation in \eqref{eq:bXpbllgkiSQgRkTr} is that we obtain the $\nicefrac{\sqrt{\ell(0)} R}{\sqrt{\varepsilon}}$ rate for small $\varepsilon$, given an appropriate or optimal choice of $\delta$ that minimizes $k(\delta).$ The main difference between Theorem~\ref{thm:convex_increasing_super_rate} and Theorem~\ref{thm:convex_increasing_rate} is that the rate in Theorem~\ref{thm:convex_increasing_super_rate} depends on $M_{\bar{R}}$ and requires its estimate.

\subsection{Example: $(\rho, L_0, L_1)$--smoothness}
\label{sec:example_smooth}
To explain how Theorem~\ref{thm:convex_increasing_super_rate} and Corollary~\ref{cor:general} work, let us consider $(\rho, L_0, L_1)$--smoothness with $\ell(x) = L_0 + L_1 x^{\rho}$ and $\rho > 0.$ In this case, $\psi(x) \simeq \frac{x^2}{L_0 + L_1 x^{\rho}},$ which is strictly increasing until $\Delta_{\max} = \infty$ if $\rho \leq 2,$ and until $\Delta_{\max} = (2 L_0 / ((\rho - 2) L_1))^{1 / \rho}$ if $\rho > 2.$ If $\rho \leq 2,$ then 
\begin{align*}
Q 
&\eqdef \left\{\delta \geq 0 \,:\, \ell(4 \psi^{-1}(\delta)) \leq 2 \ell(0), \delta \leq \Delta\right\} \\
&= \left\{\delta \geq 0 \,:\, \ell(8 \sqrt{\delta \ell\left(0\right)}) \leq 2 \ell\left(0\right), \delta \leq \Delta\right\} \\
&= \left\{\delta \geq 0 \,:\, \delta \leq L_0^{2 / \rho - 1} / (64 L_1^{2 / \rho}), \delta \leq \Delta\right\}
\end{align*}
and, using GD and the result from Table~2 by \citet{tyurin2024toward} with $\rho < 2$ to estimate $k_{\textnormal{\algname{GD}}}(\delta)$ and Theorem~\ref{thm:convex_increasing_super_rate},
\[
\resizebox{\linewidth}{!}{$
\begin{aligned}
  &\textstyle \frac{5 \sqrt{\ell(0)} \bar{R}}{\sqrt{\varepsilon}} + \min\limits_{\delta \in Q} k(\delta) \\
  &\textstyle = \cO\left(\frac{\sqrt{L_0} \bar{R}}{\sqrt{\varepsilon}} + \min\limits_{\delta \in Q} \left[\left[\log\left(\frac{\delta}{L_0 \bar{R}^2}\right)\right]_+ + \frac{L_0 \bar{R}^2}{\delta} + \frac{L_1 \Delta^{\rho / 2} \bar{R}^{2 - \rho}}{\delta^{1 - \rho / 2}}\right]\right) \\
  &\textstyle = \cO\left(\frac{\sqrt{L_0} \bar{R}}{\sqrt{\varepsilon}} + \left[\frac{L_0 \bar{R}^2}{\bar \delta} + \frac{L_1 \Delta^{\rho / 2} \bar{R}^{2 - \rho}}{{\bar \delta}^{1 - \rho / 2}}\right]\right),
\end{aligned}
$}
\]
where $\Delta \eqdef f(x^0) - f(x^*),$ and we take $\bar{R} = R$ and $\bar \delta = \min\{L_0^{2 / \rho - 1} / L_1^{2 / \rho}, L_0 \bar{R}^2, \Delta\} / 64$ to get the last complexity (which might not be the optimal choice, but a sufficient choice to show that the first term dominates if $\varepsilon$ is small). Similarly, for the case $\rho = 2,$ the oracle complexity at least 
\begin{align*}
  \textstyle \cO\left(\frac{\sqrt{L_0} R}{\sqrt{\varepsilon}} + \frac{L_0 R^2}{\bar{\delta}}+ \frac{L_1 M_0^{\rho} R^2}{\bar{\delta}}\right)
\end{align*}
with $\bar{\delta} = \min\{L_0^{2 / \rho - 1} / L_1^{2 / \rho}, L_0 R^2, \Delta\} / 64$ and $\bar{R} = R,$ where we take the \algname{GD} rate from \citep{li2024convex,tyurin2024toward}.

We now consider the case $\rho > 2.$ Let us define $\Delta_{1} \eqdef \nicefrac{1}{2}(L_0 / L_1)^{1 / \rho}.$ Notice that $\Delta_{\max} \geq \Delta_1.$ For all $\delta \in [0, \psi(\Delta_1)),$ we can find $\Delta_{\textnormal{left}}(\delta) = \psi^{-1}(\delta) \simeq \sqrt{L_0 \delta}.$ For all $x \geq \Delta_{\max},$ $\psi(x)$ is decreasing, and $\psi(x) \simeq \frac{x^2}{L_1 x^{\rho}}$ Thus, $\Delta_{\textnormal{right}}(\delta) \simeq (L_1 \delta)^{1 / (2 - \rho)}$ and we should minimize $k(\delta)$ over the set $\{\delta \in [0, L_0^{2/\rho - 1} / L_1^{2 / \rho}]: \delta \leq L_0 / L_1^2, \delta \leq (1 / (2 M_{\bar{R}}))^{\rho - 2} / L_1, \delta \leq \Delta\} \subseteq Q$ (up to constant factors). It is sufficient to take 
\begin{align}
  \label{eq:asdasdasfas}
  \textstyle \bar{\delta} \eqdef \min\left\{\frac{L_0^{2/\rho - 1}}{L_1^{2 / \rho}}, \frac{L_0}{L_1^2}, \frac{1}{\left(2 M_{\bar{R}}\right)^{\rho - 2} L_1}, L_0 \bar{R}^2, \Delta\right\}
\end{align}
to get the complexity 
\begin{align*}
  &\textstyle \cO\left(\frac{\sqrt{L_0} R}{\sqrt{\varepsilon}} + \min\limits_{\delta \in Q} k(\delta)\right) \\
  &\textstyle = \cO\left(\frac{\sqrt{L_0} R}{\sqrt{\varepsilon}} + \frac{L_0 R^2}{\bar{\delta}}+ \frac{L_1 M_0^{\rho} R^2}{\bar{\delta}}\right),
\end{align*}
where $M_0 \eqdef \norm{\nabla f(x^0)},$ $k_{\textnormal{\algname{GD}}}(\delta)$ is the rate of GD from \citep{li2024convex,tyurin2024toward}, and we take $\bar{R} = R.$ Thus, we can guarantee the $\nicefrac{\sqrt{L_0} R}{\sqrt{\varepsilon}}$ rate for any $\rho \geq 0$ and a sufficiently small $\varepsilon.$

\subsection{Discussion and proof sketch}
In the superquadratic case, we use Algorithm~\ref{alg:main} instead of Algorithm~\ref{alg:main_new} because the latter relies on the fact that $\psi$ is invertible on $\R_+.$ The former algorithm does not need this and allows us to get the $\nicefrac{\sqrt{L_0} R}{\sqrt{\varepsilon}}$ rate for small--$\varepsilon.$ While once again the proof of Theorem~\ref{thm:convex_increasing_super_rate} follows the discussion from Section~\ref{sec:proof_sketch}, there is one important difference. Since $\psi$ might not be invertible, we cannot conclude that $\norm{\nabla f(y^k)} \leq \psi^{-1}(\delta)$ if $f(y^k) - f(x^*) \leq \delta.$ Instead, we can only guarantee that if $f(y^k) - f(x^*) \leq \delta$ and $\delta \in [0, \psi\left(\Delta_{\max}\right)),$ then either $\norm{\nabla f(y^k)} \leq \Delta_{\textnormal{left}}(\delta)$ or $\norm{\nabla f(y^k)} \geq \Delta_{\textnormal{right}}(\delta),$ where $\Delta_{\max},\Delta_{\textnormal{left}}(\delta),$ and $\Delta_{\textnormal{right}}(\delta)$ are defined in Section~\ref{sec:super}. The latter case is ``bad'' for the analysis. To avoid it, we take $\delta$ such that $\Delta_{\textnormal{right}}(\delta) \geq 2 M_{\bar{R}} = 2 \max_{f(x) - f(x^*) \leq \Delta, \norm{x - x^*} \leq 2 \bar{R}} \norm{\nabla f(x)}$ and, using mathematical induction, ensure that $\norm{\nabla f(y^k)} \leq M_{\bar{R}}.$ To get the last bound, we prove that $y^k$ never leaves set $\{x \in \R^d \, | \, f(x) - f(x^*) \leq \Delta, \norm{x - x^*} \leq 2 \bar{R}\},$ which requires additional technical steps. Thus, we are left with the ``good'' case $\norm{\nabla f(y^k)} \leq \Delta_{\textnormal{left}}(\delta),$ which yields $\ell(4 \norm{\nabla f(y^k)}) \leq 2 \ell(0)$ for $\delta$ such that $\ell(4 \Delta_{\textnormal{left}}(\delta)) \leq 2 \ell(0).$

\section{Conclusion}
While we have achieved a better oracle complexity for small~$\varepsilon$, the optimal non-dominant term for large~$\varepsilon$, which can improve the terms not depending on $\varepsilon$ in Corollaries~\ref{cor:first}, \ref{cor:general} and Theorem~\ref{thm:convex_increasing_better_rate} for $\ell$--smooth functions, remains unclear and require further investigations. Moreover, it would be interesting to extend our results to stochastic and finite-sum settings \citep{schmidt2017minimizing,lan2020first}.
We leave these directions for future work, which can build on our new insights and algorithms.

\section*{Acknowledgements}
The work was supported by the grant for research centers in the field of AI provided by the Ministry of Economic Development of the Russian Federation in accordance with the agreement 000000C313925P4F0002 and the agreement №139-10-2025-033.

\section*{Impact Statement}
This paper presents work whose goal is to advance the field of Machine
Learning. There are many potential societal consequences of our work, none
which we feel must be specifically highlighted here.

\bibliography{example_paper}
\bibliographystyle{icml2026}

\newpage
\appendix
\onecolumn

\section{Experiments}

\subsection{Comparison with GD}
\label{sec:exp_one}
We compare GD \citep{vankov2024optimizing,tyurin2024toward} and AGD (Algorithm~\ref{alg:main_new}) on the function $f \,:\, \R^2 \to \R$ defined as $f(x,y) = e^x + e^{1-x} + \frac{\mu}{2} y^2,$ where $\mu = 0.001.$ This function is $(3.3 + \mu, 1)$--smooth and has its minimum at $(0.5, 0).$ Starting at $x^0 = (-6, -5)$, and taking $\bar{R} = 100 \gg R$ and $\Gamma_0 = 100 \gg \nicefrac{2 \Delta}{R^2}$ (large enough) in Algorithm~\ref{alg:main_new}, we obtain Figure~\ref{fig:fig2}. In this plot, we observe the distinctive accelerated convergence rate of Algorithm~\ref{alg:main_new} with non-monotonic behavior, supporting our theoretical results.

\begin{figure}[H]
  \centering
  \includegraphics[width=0.5\textwidth]{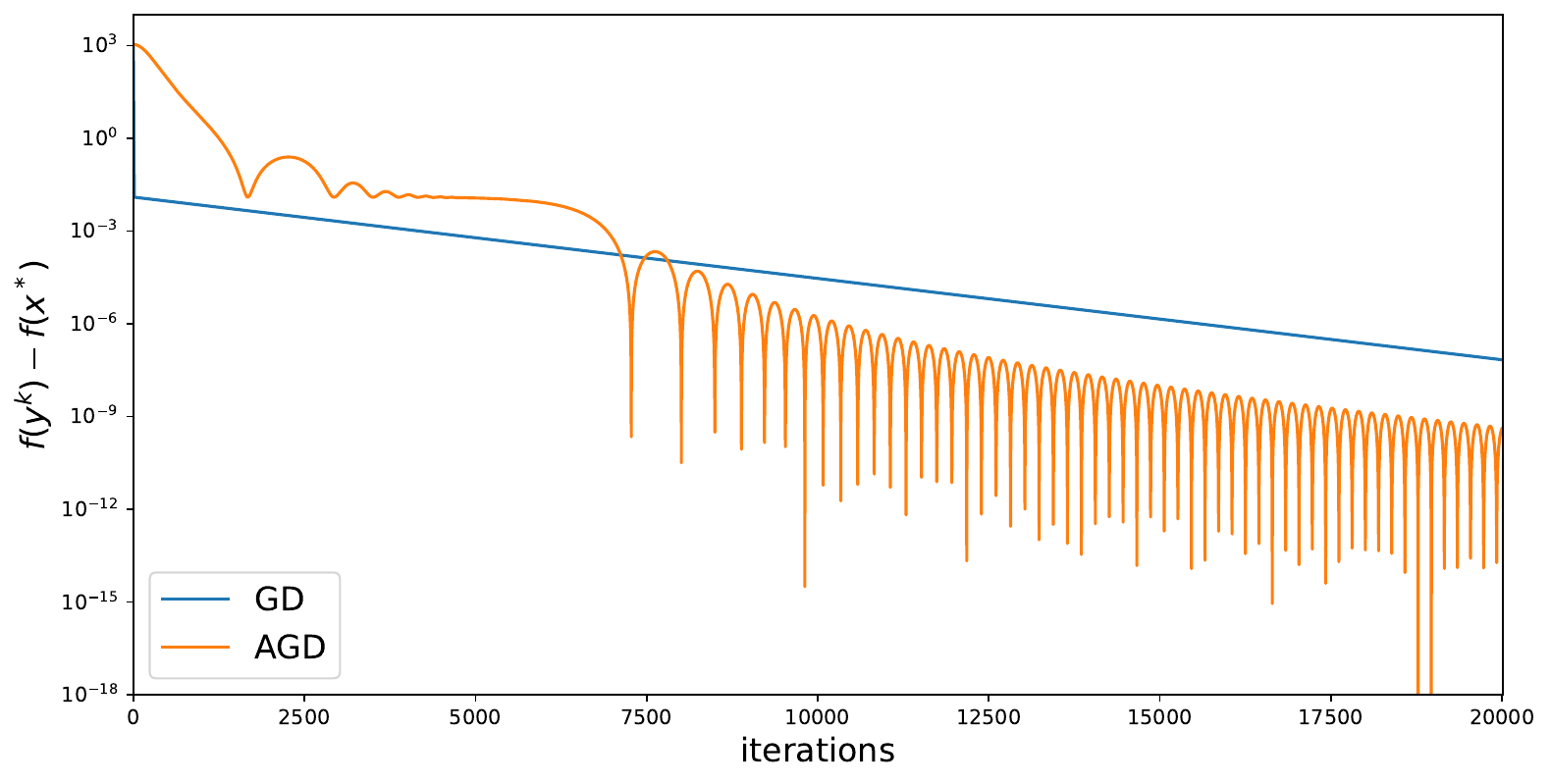}
  \caption{Experiment with $e^x + e^{1-x} + \frac{\mu}{2} y^2$ and $\mu = 0.001$}
  \label{fig:fig2}
\end{figure}

\subsection{Comparison with previous AGD methods}
Using the same function and setup, we compare our Algorithm~\ref{alg:main_new} with previous accelerated methods in Figure~\ref{fig:fig3}. For all methods, we choose parameter values according to the theorems in their respective papers. Notice that AGD by \citet{vankov2024optimizing} requires a method that solves an auxiliary problem. To solve this problem, we use binary search with $10$ and $100$ steps. In Figure~\ref{fig:fig3}, we observe very different behaviors across the methods. AGD by \citet{li2024convex} has the slowest convergence since their method chooses a small step size. The method by \citet{vankov2024optimizing} is sensitive to the number of inner steps used to solve the auxiliary problem: with only inner step $10$ steps, it converges slowly. At the beginning, the method by \citet{gorbunov2024methods} has the fastest convergence, while our method performs better at lower accuracies.

\begin{figure}[H]
  \centering
  \includegraphics[width=0.5\textwidth]{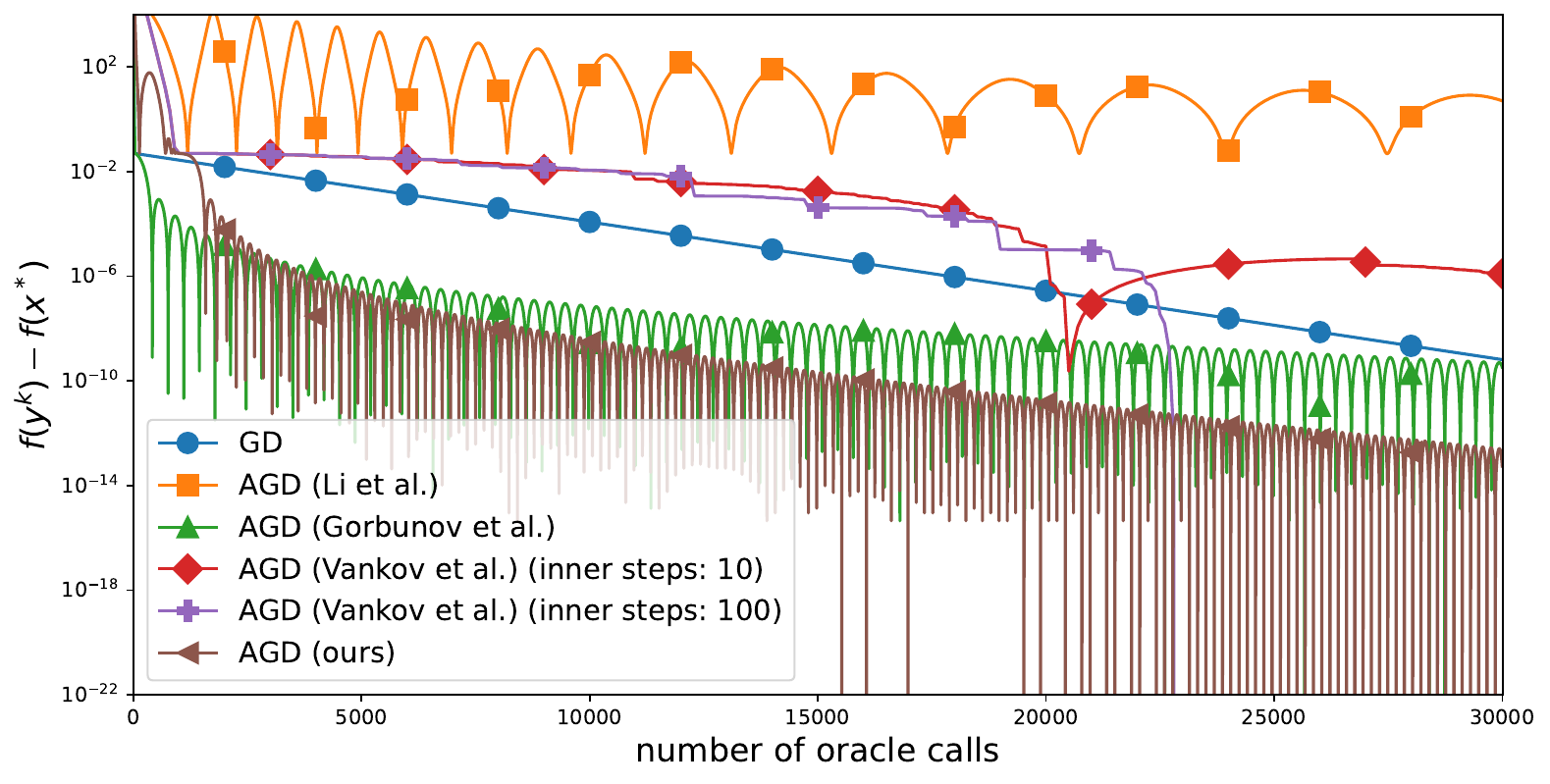}
  \caption{Experiment with $e^x + e^{1-x} + \frac{\mu}{2} y^2$ and $\mu = 0.001$}
  \label{fig:fig3}
\end{figure}

\subsection{Sensitivity to the choice of $\bar{R}$ and $\Gamma_0$}
We now also check how sensitive our algorithm is to the choice of $\bar{R}$ and $\Gamma_0$. In Figures~\ref{fig:fig_increasing_Gamma0} and \ref{fig:fig_increasing_R}, we fix the theoretically best values and increase them by $5\times$ and $25\times$. We observe that the algorithm is not very sensitive to the choice of $\Gamma_0,$ but more sensitive to the choice of $\bar{R},$ which is expected since $\Gamma_0$ is under the logarithms in \eqref{eq:jKlCMcGdKddup}, while $\bar{R}$ is not.

\begin{figure}[H]
  \centering
  \includegraphics[width=0.5\textwidth]{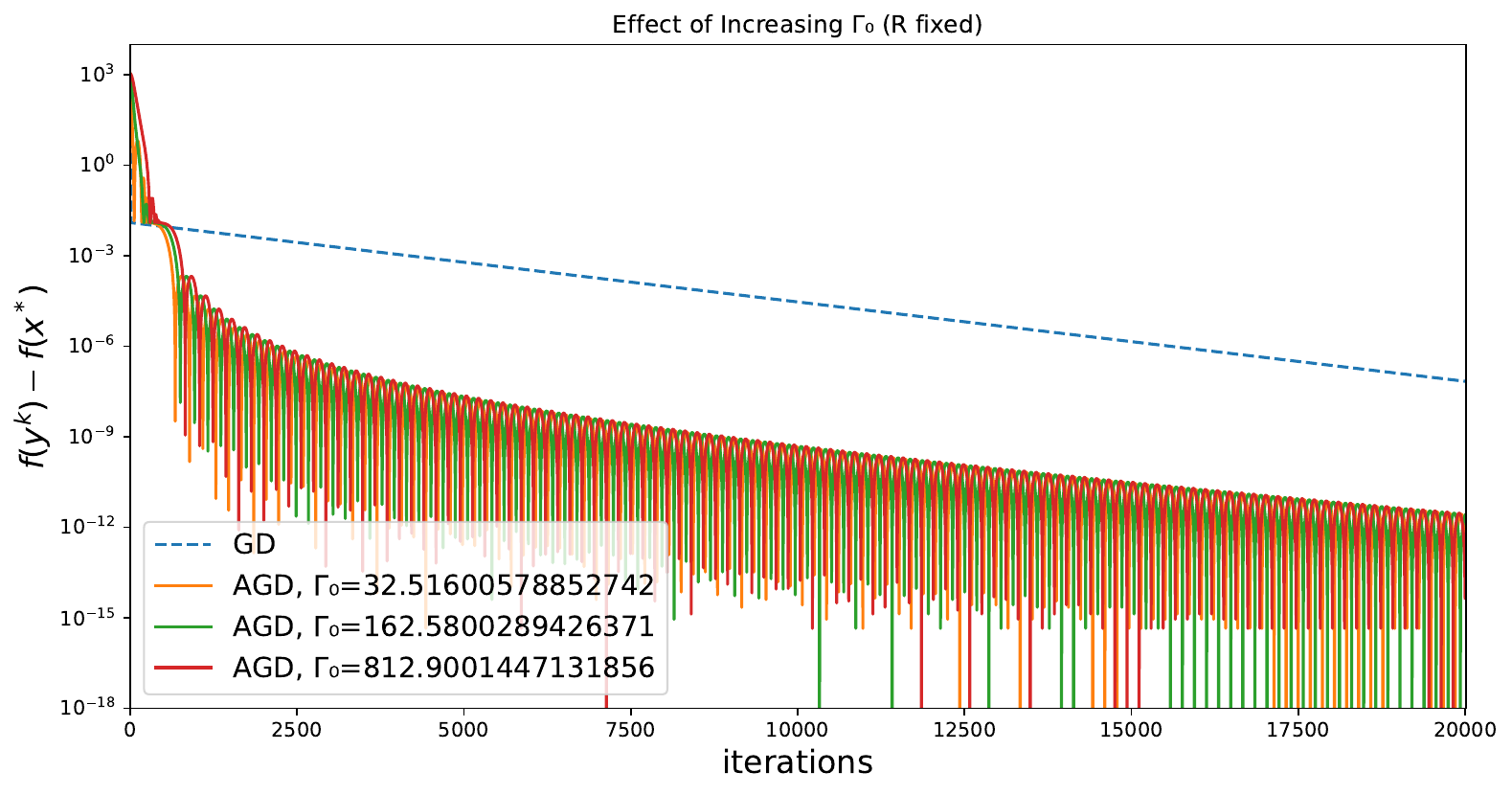}
  \caption{Sensitivity to increasing $\Gamma_0$ by $5\times$ and $25\times$.}
  \label{fig:fig_increasing_Gamma0}
\end{figure}

\begin{figure}[H]
  \centering
  \includegraphics[width=0.5\textwidth]{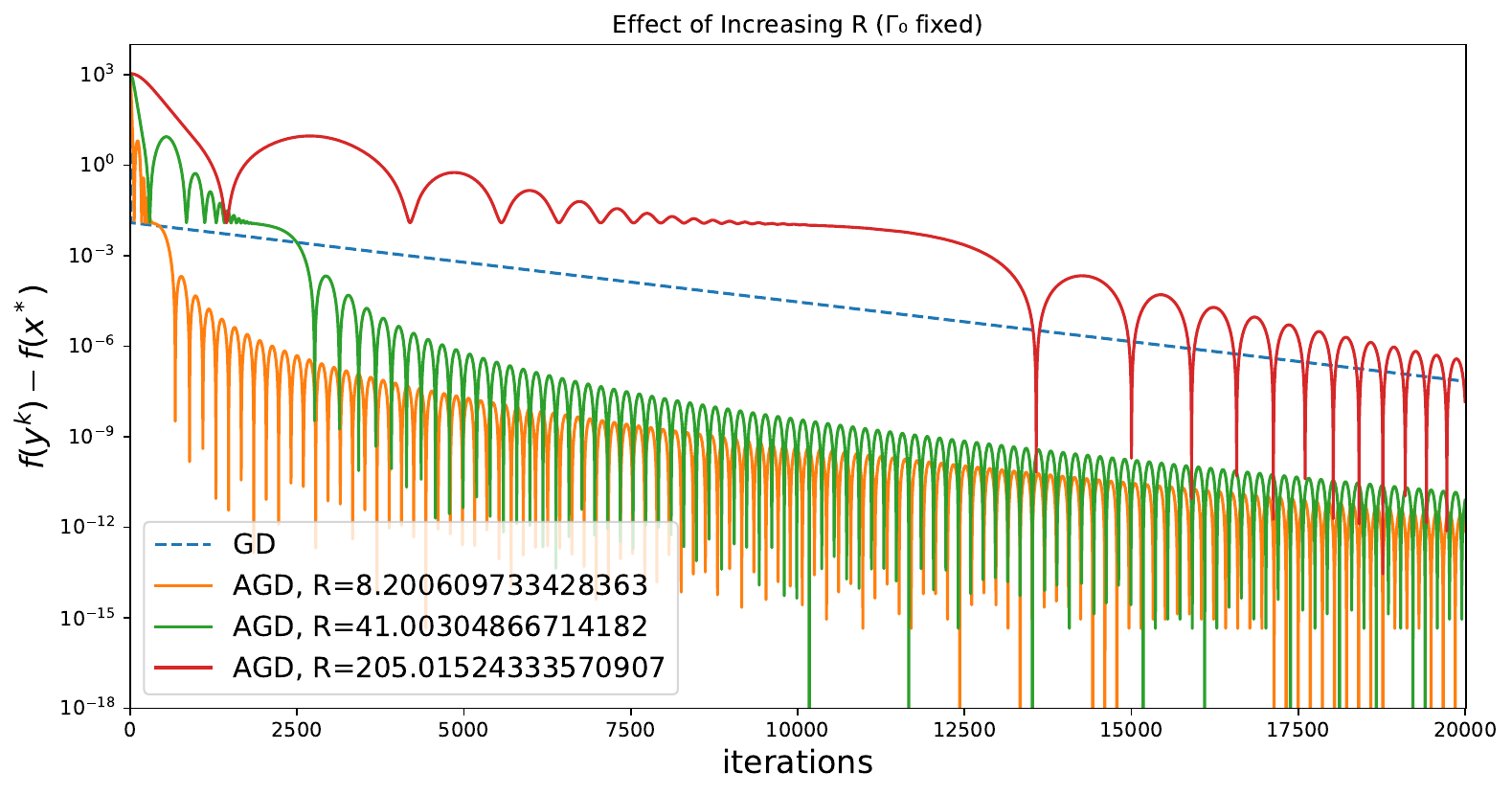}
  \caption{Sensitivity to increasing $\bar{R}$ by $5\times$ and $25\times$.}
  \label{fig:fig_increasing_R}
\end{figure}

\subsection{Experiments with Algorithm~\ref{alg:main} and non-monotonic $\psi$}
We now consider Algorithm~\ref{alg:main} and the results from Section~\ref{sec:super}. We take the function $f \,:\, \R^2 \to \R$ defined as $f(x,y) = -\sqrt{x} - \sqrt{1 - x} + \frac{\mu}{2} y^2,$ where $\mu = 0.001,$ which is $(3, 4, 10)$–smooth. We start at $x^0 = (0.3, -0.15)$ and take $\bar{R} = R$ in Algorithm~\ref{alg:main}. Unlike Algorithm~\ref{alg:main_new}, we have to choose $\delta.$ We can take $M_{\bar{R}} = 4.47 \geq \max_{f(x) - f(x^*) \leq \Delta, \norm{x - x^*} \leq 2 \bar{R}} \norm{\nabla f(x)},$ which we estimated numerically. Then, we choose $\delta$ according to \eqref{eq:asdasdasfas}, where the latter choice was derived for $(\rho, L_0, L_1)$–smooth functions. The results are presented in Figure~\ref{fig:fig_non_monotonic}. In practice, we observe that the required number of GD steps is small, less than $ 10$, and thus the GD iterations in Algorithm~\ref{alg:main} are almost invisible in the plot. Similarly to Section~\ref{sec:exp_one}, AGD converges non-monotonically faster than GD.

\begin{figure}[H]
  \centering
  \includegraphics[width=0.5\textwidth]{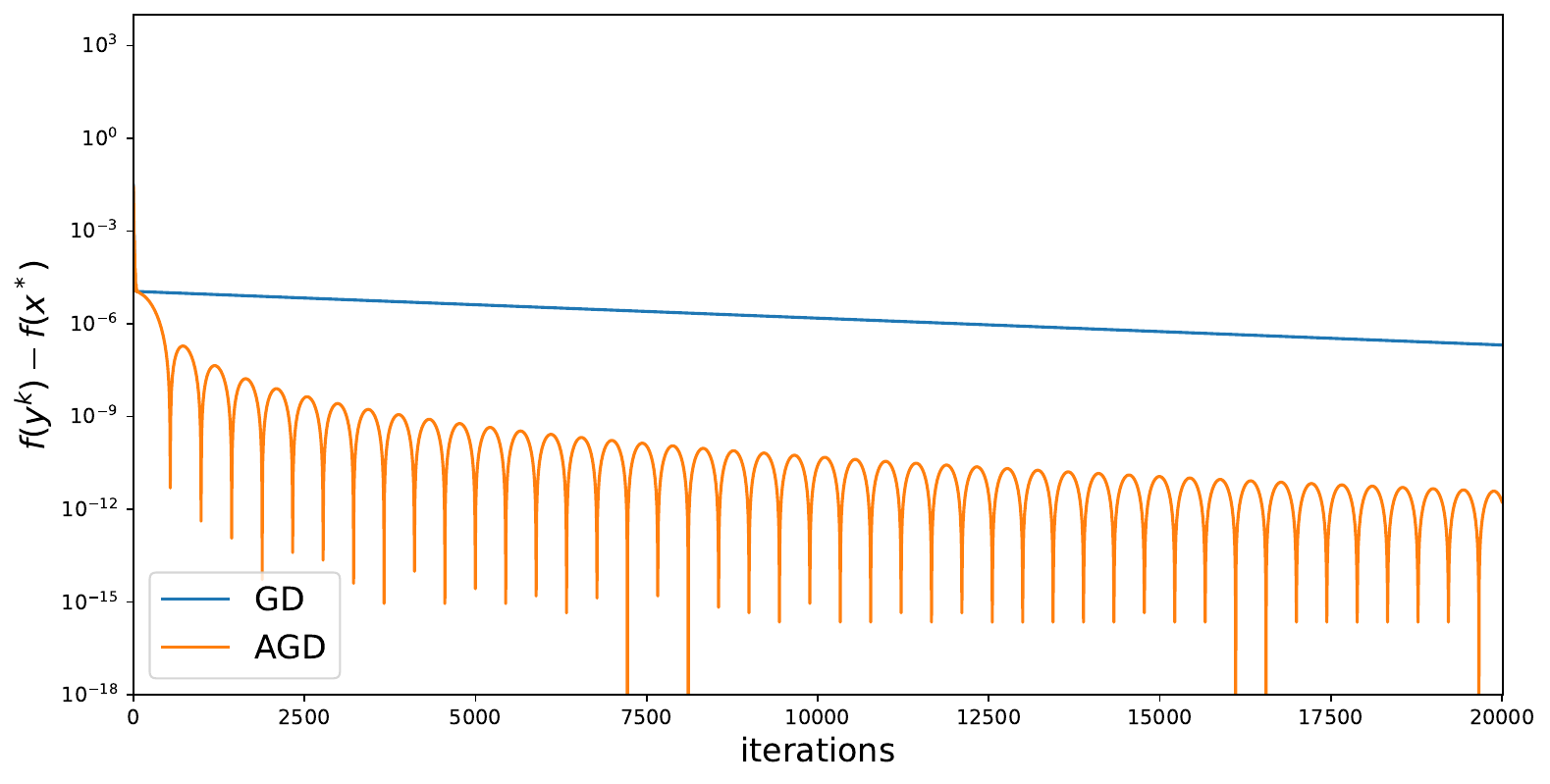}
  \caption{Experiment with $-\sqrt{x} - \sqrt{1 - x} + \frac{\mu}{2} y^2$ and $\mu = 0.001$}
  \label{fig:fig_non_monotonic}
\end{figure}

\subsection{Experiments with classifaction problem}

We conduct additional experiments on MNIST \citep{lecun2010mnist} multiclass logistic regression with a cubic regularizer $+ \lambda \|\cdot\|^3,$ where $\lambda = 0.0001.$ We include the cubic regularizer to ensure that the objective is not $L$-smooth but $(L_0, L_1)$-smooth. We evaluate all methods on MNIST, where each method optimizes the same objective with $L_1 = 1$, and $L_0$ is set to the logistic smoothness bound plus $1$, which guarantees that the objective is $(L_0, L_1)$-smooth. In Algorithm~\ref{alg:main_new}, we take $\bar{R} = 0.1$ and $\Gamma_0 = 1.$

We report both training loss and test accuracy trajectories in Figure~\ref{fig:ml}. The experimental results are consistent with our theory, and we observe that our method demonstrates good practical performance. It is also worth mentioning that the method of \citet{gorbunov2024methods} exhibits strong practical performance, despite having weaker theoretical guarantees.

\begin{figure}[h]
  \centering
  \includegraphics[width=\textwidth]{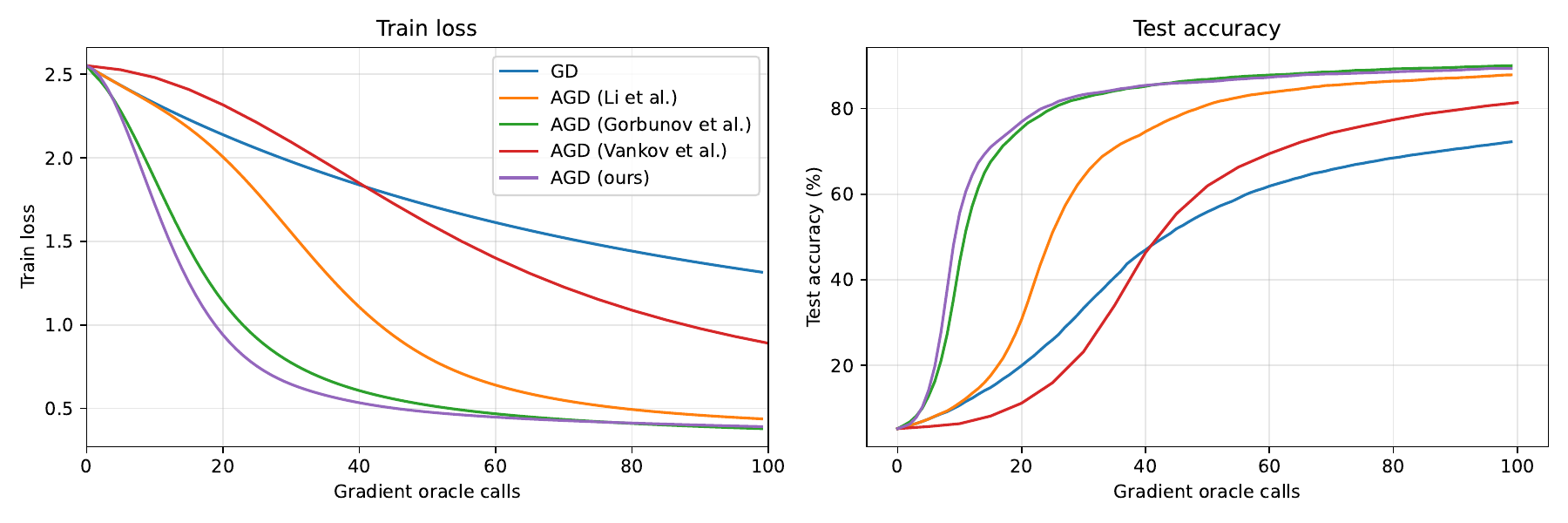}
  \caption{Experiment with logistic regression and cubic regularizer.}
  \label{fig:ml}
\end{figure}

\newpage
\section{Auxiliary Lemmas}

In the proofs, we use the following useful lemma from \citep{tyurin2024toward}, which generalizes the key inequality from Theorem 2.1.5 of \citep{nesterov2018lectures}.
\begin{lemma}[\citet{tyurin2024toward}]
  \label{lemma:smooth_convex}
  For all $x, y \in \mathcal{X},$ if $f$ is $\ell$--smooth (Assumption~\ref{ass:gen_smooth}) and convex (Assumption~\ref{ass:convex}), then
  \begin{align}
    \label{eq:gen_1}
    \norm{\nabla f(x) - \nabla f(y)}^2 \int_{0}^{1} \frac{1 - v}{\ell(\norm{\nabla f(x)} + \norm{\nabla f(x) - \nabla f(y)} v)} d v \leq f(x) - f(y) - \inp{\nabla f(y)}{x - y}.
  \end{align}
\end{lemma}

The following lemma ensures that it is ``safe'' to take steps with proper step sizes.
\begin{lemma}[\citet{tyurin2024toward}]
  \label{lemma:aux_2}
  Under Assumption~\ref{ass:gen_smooth}, for a fixed $x \in \mathcal{X},$ the point $y = x + t h \in \mathcal{X}$ for all $t \in \left[0, \int_{0}^{\infty} \frac{d v}{\ell(\norm{\nabla f(x)} + v)}\right)$ and $h \in \R^d$ such that $\norm{h} = 1.$
\end{lemma}

We now prove two important lemmas that allow us to bound the norm $\norm{\nabla f(y)}$ given an upper bound on $f(y) - f(x^*)$.

\begin{restatable}{lemma}{LEMMABOUND}[Strictly Increasing $\psi$]
  \label{lemma:func_to_grad_spec}
  Under Assumptions~\ref{ass:gen_smooth} and \ref{ass:convex}, let $f(y) - f(x^*) \leq \delta$ for some $y \in \mathcal{X}, \delta > 0$ and $\psi \,:\, \R_{+} \to \R_{+}$ such that $\psi(x) = \frac{x^2}{2 \ell(4 x)}$ is strictly increasing, then $\norm{\nabla f(y)} \leq \psi^{-1}(\delta)$ if $\delta \in \textnormal{im} (\psi).$
\end{restatable}
\begin{proof}
  Using Lemma~\ref{lemma:smooth_convex} and the fact that $\ell$ is non-decreasing,
  \begin{align*}
    \delta 
    &\geq f(y) - f(x^*) 
    \geq \norm{\nabla f(y)}^2 \int_{0}^{1} \frac{1 - v}{\ell(\norm{\nabla f(y)} + \norm{\nabla f(y)} v)} d v \\
    &\geq \frac{\norm{\nabla f(y)}^2}{2 \ell(4 \norm{\nabla f(y)})} = \psi\left(\norm{\nabla f(y)}\right).
  \end{align*}
  It is left to invert $\psi$ to get the result.
\end{proof}

\begin{restatable}{lemma}{LEMMABOUNDDIFF}[Not Necessarily Strictly Increasing $\psi$]
  \label{lemma:func_to_grad_spec_diff}
  Under Assumptions~\ref{ass:gen_smooth} and \ref{ass:convex}, let $\psi \,:\, \R_{+} \to \R_{+}$ such that $\psi(x) = \frac{x^2}{2 \ell(4 x)}$ is not necessarily strictly increasing.
  \begin{enumerate}
    \item There exists the largest $\Delta_{\max} \in (0, \infty]$ such that $\psi$ is strictly increasing on $[0, \Delta_{\max})$,
    \item For all $\delta \in [0, \psi(\Delta_{\max})),$ there exists the unique $\Delta_{\textnormal{left}}(\delta) \in [0, \Delta_{\max})$ and the smallest\footnote{if the set $\{x \in [\Delta_{\max}, \infty) \,:\, \psi(x) = \delta\}$ is empty, then $\Delta_{\textnormal{right}}(\delta) = \infty$} $\Delta_{\textnormal{right}}(\delta) \in [\Delta_{\max}, \infty]$ such that $\psi(\Delta_{\textnormal{left}}(\delta)) = \delta$ and $\psi(\Delta_{\textnormal{right}}(\delta)) = \delta.$ 
    \item For all $\delta \in [0, \psi(\Delta_{\max})),$ if $\Delta_{\textnormal{right}}(\delta) < \infty$ and $\delta > \bar{\delta} \geq 0$, then $\Delta_{\textnormal{right}}(\bar{\delta}) > \Delta_{\textnormal{right}}(\delta).$
    \item If $f(y) - f(x^*) \leq \delta$ for some $y \in \mathcal{X}$ and $\delta \in [0, \psi\left(\Delta_{\max}\right)),$ then either $\norm{\nabla f(y)} \leq \Delta_{\textnormal{left}}(\delta)$ or $\norm{\nabla f(y)} \geq \Delta_{\textnormal{right}}(\delta).$
  \end{enumerate}
\end{restatable}

\begin{proof}
  1. Since $\ell$ is non-decreasing and locally Lipschitz, there exists $\bar{\Delta}_1 > 0$ such that
  \begin{align*}
    2 \ell(4 y) - 2 \ell(4 x) \leq M (y - x)
  \end{align*}
  for all $0 \leq x < y \leq \bar{\Delta}_1$ and for some $M \equiv M(\bar{\Delta}_1, \ell) > 0.$ Thus,
  \begin{align}
    \label{eq:QYtfMbLCGaG}
    x^2 2 \ell(4 y) \leq x^2 2 \ell(4 x) + M x^2 (y - x).
  \end{align}
  Moreover, there exists $\bar{\Delta}_2 > 0$ such that
  \begin{align*}
    M x^2 < (y + x) 2 \ell(4 x)
  \end{align*}
  for all $0 \leq x < y \leq \bar{\Delta}_2$ since $2 \ell(4 x) \geq \ell(0) > 0,$ the l.h.s $\cO(x^2),$ and the r.h.s. $\Omega(x).$ Combining with \eqref{eq:QYtfMbLCGaG},
  \begin{align*}
    x^2 2 \ell(4 y) < x^2 2 \ell(4 x) + 2 \ell(4 x) (y + x) (y - x) = y^2 2 \ell(4 x)
  \end{align*}
  and 
  \begin{align*}
    \frac{x^2}{2 \ell(4 x)} < \frac{y^2}{2 \ell(4 y)}
  \end{align*}
  for all $0 \leq x < y \leq \min\{\bar{\Delta}_1,\bar{\Delta}_2\},$ meaning that $\psi$ is locally strictly increasing on the interval $[0, \Delta_{\max})$ for some largest $\Delta_{\max} \in (0, \infty].$ 

  2. $\Delta_{\textnormal{left}}(\delta)$ exists since $\psi$ is locally strictly increasing on the interval $[0, \Delta_{\max}).$ On the interval $[\Delta_{\max}, \infty),$ either $\psi$ intersects $\delta$ for the first time at $\Delta_{\textnormal{right}}(\delta)$ or we can take $\Delta_{\textnormal{right}}(\delta) = \infty.$

  3. Since $\Delta_{\textnormal{right}}(\delta)$ is the first time when $\psi$ intersects $\delta$ for $x \in [\Delta_{\max}, \infty)$ and $\delta < \psi(\Delta_{\max}),$ then $\psi(x) > \delta$ for all $x \in [\Delta_{\max}, \Delta_{\textnormal{right}}(\delta)).$ Thus, if we decrease $\delta$ and take $\bar{\delta} < \delta,$ then $\Delta_{\textnormal{right}}(\bar{\delta})$ can only increase or stay the same. However, if $\Delta_{\textnormal{right}}(\bar{\delta})$ stays the same, i.e., $\Delta_{\textnormal{right}}(\bar{\delta}) = \Delta_{\textnormal{right}}(\delta),$ then $\Delta_{\textnormal{right}}(\bar{\delta})$ is the first time when $\psi$ intersects $\delta,$ which is impossible due to the continuity of $\psi$ and the fact that $\Delta_{\textnormal{right}}(\bar{\delta})$ is the first time when $\psi$ intersects $\bar{\delta} < \delta.$
  
  4. Using the same reasoning as in the proof of Lemma~\ref{lemma:func_to_grad_spec}:
  \begin{align}
    \label{eq:iqtxtpBzUBCLZuuHF}
    \delta \geq \psi\left(\norm{\nabla f(y)}\right).
  \end{align}
  Due to the previous properties, either $\norm{\nabla f(y)} \leq \Delta_{\textnormal{left}}(\delta)$ or $\norm{\nabla f(y)} \geq \Delta_{\textnormal{right}}(\delta)$ because $\psi\left(x\right) > \delta$ for all $x \in (\Delta_{\textnormal{left}}(\delta), \Delta_{\textnormal{right}}(\delta)).$
\end{proof}

\section{Rate of the Auxiliary Sequence}

\THEOREMSEQ*

\begin{proof}
  By the definition of $\Gamma_{k+1}$ and $\alpha_k,$
  \begin{align*}
    \Gamma_{k+1} \leq \frac{\Gamma_{k}}{1 + \sqrt{\gamma \Gamma_k}}
  \end{align*}
  for all $k \geq 0.$ Instead of $\Gamma_{k},$ consider the sequence $\bar{\Gamma}_{k}$ such that 
  \begin{align*}
    \bar{\Gamma}_{k+1} = \frac{\bar{\Gamma}_{k}}{1 + \sqrt{\gamma \bar{\Gamma}_k}}
  \end{align*}
  for all $k \geq 0$ and $\bar{\Gamma}_0 = \Gamma_0.$ Using mathematical induction, notice that $\bar{\Gamma}_{k+1} \geq \Gamma_{k+1}.$ Indeed, the function $\frac{x}{1 + \sqrt{\gamma x}}$ is increasing\footnote{$(\frac{x}{1 + \sqrt{\gamma x}})' = \frac{1 + \frac{\sqrt{\gamma x}}{2}}{(1 + \sqrt{\gamma x})^2} > 0$ for all $x \geq 0.$} for all $x \geq 0$ and 
  \begin{align*}
    \Gamma_{k+1} \leq \frac{\Gamma_{k}}{1 + \sqrt{\gamma \Gamma_k}} \leq \frac{\bar{\Gamma}_{k}}{1 + \sqrt{\gamma \bar{\Gamma}_k}} = \bar{\Gamma}_{k+1}
  \end{align*}
  if $\Gamma_k \leq \bar{\Gamma}_{k}.$ If we bound $\bar{\Gamma}_{k+1},$ then we can bound $\Gamma_{k+1}.$ Next,
  \begin{align*}
     \frac{1}{\bar{\Gamma}_{k+1}} - \frac{1}{\bar{\Gamma}_{k}} = \sqrt{\frac{\gamma}{\bar{\Gamma}_k}}
  \end{align*}
  Let us define $t_k \eqdef \frac{1}{\bar{\Gamma}_{k}}$ for all $k \geq 0,$ then 
  \begin{align}
    \label{eq:RKlLSXvVgDKl}
    t_{k+1} - t_k = \sqrt{\gamma t_k}.
  \end{align}
  and
  \begin{align}
    \label{eq:yyugvGKCzBdZypgh}
    (t_{k+1}^{1/2} + t_k^{1/2}) (t_{k+1}^{1/2} - t_k^{1/2}) = \sqrt{\gamma t_k}
  \end{align}
  for all $k \geq 0.$ We now fix any $k \geq 0.$ There are two options: \\
  \textbf{Option 1:} $t_k^{1/2} \leq \frac{\sqrt{\gamma}}{2}.$\\
  In this case, using \eqref{eq:RKlLSXvVgDKl},
  \begin{align*}
    t_{k+1} = t_k + \sqrt{\gamma t_k} \leq \frac{\gamma}{4} + \frac{\gamma}{2} = \frac{3 \gamma}{4}
  \end{align*}
  and 
  \begin{align*}
    2 \sqrt{\gamma} (t_{k+1}^{1/2} - t_k^{1/2}) \geq \sqrt{\gamma t_k}
  \end{align*}
  due to \eqref{eq:yyugvGKCzBdZypgh}. Rearranging the terms,
  \begin{align}
    \label{eq:NrGBV}
    t_{k+1}^{1/2} \geq \frac{3}{2} t_k^{1/2} \geq \left(\frac{3}{2}\right)^{k+1} t_0^{1/2},
  \end{align}
  where we unroll the recursion since $t_0^{1/2} \leq \dots \leq t_k^{1/2} \leq \frac{\sqrt{\gamma}}{2}.$ \\
  \textbf{Option 2:} $t_k^{1/2} > \frac{\sqrt{\gamma}}{2}.$\\
  Using \eqref{eq:RKlLSXvVgDKl},
  \begin{align*}
    t_{k+1} = t_k + \sqrt{\gamma t_k} \leq t_k + 2 t_k \leq 3 t_k
  \end{align*}
  and 
  \begin{align*}
    3 t_{k}^{1/2} (t_{k+1}^{1/2} - t_k^{1/2}) \geq \sqrt{\gamma t_k}
  \end{align*}
  due to \eqref{eq:yyugvGKCzBdZypgh}, which yields
  \begin{align}
    \label{eq:WGutGvSTvAVMNXNw}
    t_{k+1}^{1/2} \geq t_k^{1/2} + \frac{\sqrt{\gamma}}{3}.
  \end{align}
  Let $k^* \geq 0$ be the smallest index such that $t_{k^*}^{1/2} > \frac{\sqrt{\gamma}}{2}.$ Unrolling \eqref{eq:WGutGvSTvAVMNXNw},
  \begin{align}
    \label{eq:JQdlZWOsKJJzCuO}
    t_{k+1}^{1/2} \geq t_{k^*}^{1/2} + (k + 1 - k^*)\frac{\sqrt{\gamma}}{3}
  \end{align}
  for all $k \geq k^*.$
  If $k^* = 0,$ then
  \begin{align}
    t_{k+1}^{1/2} \geq (k + 1)\frac{\sqrt{\gamma}}{3}.
  \end{align}
  Otherwise, by the definition of $k^*,$
  \begin{align*}
    \left(\frac{3}{2}\right)^{k^* - 1} t_0^{1/2} \overset{\eqref{eq:NrGBV}}{\leq} t_{k^* - 1}^{1/2} \leq \frac{\sqrt{\gamma}}{2},
  \end{align*}
  which yields 
  \begin{align*}
    k^* \leq 1 + \frac{1}{2} \log_{3/2}\left(\frac{\gamma}{4 t_0}\right)
  \end{align*}
  and 
  \begin{align}
    t_{k+1}^{1/2} \geq \left(k + 1 - \left(1 + \frac{1}{2} \log_{3/2}\left(\frac{\gamma}{4 t_0}\right)\right)\right) \frac{\sqrt{\gamma}}{3},
  \end{align}
  due to \eqref{eq:JQdlZWOsKJJzCuO}. Combining the cases with $k^* = 0$ and $k^* > 0,$ we get
  \begin{align}
    t_{k+1}^{1/2} \geq \left(k + 1 - \bar{k}\right) \frac{\sqrt{\gamma}}{3}
  \end{align}
  for all $k \geq \bar{k} \eqdef \max\left\{1 + \frac{1}{2} \log_{3/2}\left(\frac{\gamma}{4 t_0}\right), 0\right\}.$ It left to recall that $t_k = 1 / \bar{\Gamma}_k$ and $\bar{\Gamma}_k \geq \Gamma_k$ for all $k \geq 0$ to obtain the result.
  
\end{proof}

\section{Main Descent Lemma}

\begin{lemma}
  \label{thm:main_lemma}
  Suppose that Assumptions~\ref{ass:gen_smooth} and \ref{ass:convex} hold. Consider Algorithm~\ref{alg:main} up to the $k$\textsuperscript{th} iteration and the following virtual steps:
  \begin{equation}
  \label{eq:TqfxanGJGnaQsmS}
  \begin{aligned}
    \alpha_k(\gamma) \equiv \alpha_{k,\gamma} &= \sqrt{\gamma \Gamma_k}, \\
    y^{k+1}(\gamma) \equiv y^{k+1}_{\gamma} &= \frac{1}{1 + \alpha_{k,\gamma}} y^{k} + \frac{\alpha_{k,\gamma}}{1 + \alpha_{k,\gamma}} u^{k} - \frac{\gamma}{1 + \alpha_{k,\gamma}} \nabla f(y^{k}), \\
    u^{k+1}(\gamma) \equiv u^{k+1}_{\gamma} &= \textnormal{Proj}_{\bar{\mathcal{X}}} \left(u^k - \frac{\alpha_{k,\gamma}}{\Gamma_k} \nabla f(y^{k+1}_{\gamma})\right), \\
    \Gamma_{k+1}(\gamma) \equiv \Gamma_{k+1, \gamma} &= \Gamma_{k} / (1 + \alpha_{k,\gamma}),
  \end{aligned}
  \end{equation}
  where $0 \leq \gamma \leq \frac{1}{\ell(2 \norm{\nabla f(y^k)})}$ is a free parameter, $y^{k} \in \mathcal{X},$ and $u^{k} \in \bar{\mathcal{X}}.$  Then, the steps \eqref{eq:TqfxanGJGnaQsmS} are well-defined, $y^{k+1}_\gamma \in \mathcal{X}$, and $u^{k+1}_{\gamma} \in \bar{\mathcal{X}},$ and
  \begin{align*}
    &(1 + \alpha_{k,\gamma})(f(y^{k+1}_{\gamma}) - f(x^*)) + \frac{(1 + \alpha_{k,\gamma}) \Gamma_{k+1, \gamma}}{2} \norm{u^{k+1}_{\gamma} - x^*}^2 - \left((f(y^{k}) - f(x^*)) + \frac{\Gamma_{k}}{2} \norm{u^{k} - x^*}^2\right) \nonumber \\
    &\leq \frac{1}{2} \left(\gamma - \frac{1}{\ell(2 \norm{\nabla f(y^{k})} + \norm{\nabla f(y^{k + 1}_{\gamma})})}\right) \norm{\nabla f(y^{k+1}_{\gamma}) - \nabla f(y^{k})}^2.
  \end{align*}
\end{lemma}

\begin{proof}
  (The following steps up to \eqref{eq:BqzoqnOOOORNZriziJ} may be skipped by the reader if $\mathcal{X} = \R^n$) \\
  Clearly, $u^{k+1}_{\gamma} \in \bar{\mathcal{X}}$ due the projection operator. However, we have to check that $y^{k+1}_\gamma \in \mathcal{X}$ to make sure the steps are well-defined. Notice that 
  \begin{align*}
    y^{k+1}_{\gamma} = \frac{1}{1 + \alpha_{k,\gamma}} \left(y^{k} - \gamma \nabla f(y^{k})\right) + \frac{\alpha_{k,\gamma}}{1 + \alpha_{k,\gamma}} u^{k}
  \end{align*}
  Moreover, $y^{k} - \gamma \nabla f(y^{k}) \in \mathcal{X}.$ If $\nabla f(y^{k}) = 0,$ then it is trivial. Otherwise,  
  \begin{align*}
    y^{k} - \gamma \nabla f(y^{k}) = y^{k} - \gamma \norm{\nabla f(y^{k})} \frac{\nabla f(y^{k})}{\norm{\nabla f(y^{k})}} \in \mathcal{X}
  \end{align*}
  due to Lemma~\ref{lemma:aux_2} because 
  \begin{align*}
    \gamma \norm{\nabla f(y^{k})} \leq \frac{\norm{\nabla f(y^{k})}}{\ell(2 \norm{\nabla f(y^{k})})} \leq \int_{0}^{\infty} \frac{dv}{\ell(\norm{\nabla f(y^{k})} + v)}.
  \end{align*}
  for all $\gamma \leq \frac{1}{\ell(2 \norm{\nabla f(y^k)})}.$ In total, $y^{k+1}_{\gamma} \in \mathcal{X}$ since $\mathcal{X}$ is an open convex set, $u^k \in \bar{\mathcal{X}},$ and $\frac{1}{1 + \alpha_{k,\gamma}} \neq 0$ (as it is a convex combination of a point from $\mathcal{X}$ and a point from $\bar{\mathcal{X}}$ with a non-zero weight; see \citep{rockafellar2015convex}[Theorem 6.1]).

  Consider the difference 
  \begin{align}
    \label{eq:BqzoqnOOOORNZriziJ}
    f(y^{k+1}_{\gamma}) - f(x^*) + \frac{\Gamma_{k+1, \gamma}}{2} \norm{u^{k+1}_{\gamma} - x^*}^2 - \left((f(y^{k}) - f(x^*)) + \frac{\Gamma_{k}}{2} \norm{u^{k} - x^*}^2\right).
  \end{align}
  Rearranging the terms, we get
  \begin{align*}
    &f(y^{k+1}_{\gamma}) - f(x^*) + \frac{\Gamma_{k+1, \gamma}}{2} \norm{u^{k+1}_{\gamma} - x^*}^2 - \left((f(y^{k}) - f(x^*)) + \frac{\Gamma_{k}}{2} \norm{u^{k} - x^*}^2\right) \\
    &= - (f(y^{k}) - f(y^{k+1}_{\gamma}) - \inp{\nabla f(y^{k+1}_{\gamma})}{y^{k} - y^{k+1}_{\gamma}}) \\
    &\quad + \inp{\nabla f(y^{k+1}_{\gamma})}{y^{k+1}_{\gamma} - y^{k}} \\
    &\quad + \frac{\Gamma_{k+1, \gamma} - \Gamma_{k}}{2} \norm{u^{k+1}_{\gamma} - x^*}^2 + \frac{\Gamma_{k}}{2}\left(\norm{u^{k+1}_{\gamma} - x^*}^2 - \norm{u^{k} - x^*}^2\right).
  \end{align*}
  Since $\Gamma_{k} = (1 + \alpha_{k,\gamma}) \Gamma_{k+1, \gamma},$
  \begin{align*}
    &f(y^{k+1}_{\gamma}) - f(x^*) + \frac{(1 + \alpha_{k,\gamma}) \Gamma_{k+1, \gamma}}{2} \norm{u^{k+1}_{\gamma} - x^*}^2 - \left((f(y^{k}) - f(x^*)) + \frac{\Gamma_{k}}{2} \norm{u^{k} - x^*}^2\right) \\
    &= - (f(y^{k}) - f(y^{k+1}_{\gamma}) - \inp{\nabla f(y^{k+1}_{\gamma})}{y^{k} - y^{k+1}_{\gamma}}) \\
    &\quad + \inp{\nabla f(y^{k+1}_{\gamma})}{y^{k+1}_{\gamma} - y^{k}} \\
    &\quad + \frac{\Gamma_{k}}{2}\left(\norm{u^{k+1}_{\gamma} - x^*}^2 - \norm{u^{k} - x^*}^2\right).
  \end{align*}
  Due to $\norm{a}^2 - \norm{a + b}^2 = - \norm{b}^2 - 2 \inp{a}{b}$ for all $a, b \in \R^d,$
  \begin{equation}
  \label{eq:xrmEVqbJM}
  \begin{aligned}
    &f(y^{k+1}_{\gamma}) - f(x^*) + \frac{(1 + \alpha_{k,\gamma}) \Gamma_{k+1, \gamma}}{2} \norm{u^{k+1}_{\gamma} - x^*}^2 - \left((f(y^{k}) - f(x^*)) + \frac{\Gamma_{k}}{2} \norm{u^{k} - x^*}^2\right) \\
    &= - (f(y^{k}) - f(y^{k+1}_{\gamma}) - \inp{\nabla f(y^{k+1}_{\gamma})}{y^{k} - y^{k+1}_{\gamma}}) \\
    &\quad + \inp{\nabla f(y^{k+1}_{\gamma})}{y^{k+1}_{\gamma} - y^{k}} \\
    &\quad + \frac{\Gamma_{k}}{2}\left(- \norm{u^{k} - u^{k+1}_{\gamma}}^2 - 2 \inp{u^{k+1}_{\gamma} - x^*}{u^{k} - u^{k+1}_{\gamma}}\right).
  \end{aligned}
  \end{equation}
  Consider the last inner product:
  \begin{align*}
    &-\inp{u^{k+1}_{\gamma} - x^*}{u^{k} - u^{k+1}_{\gamma}} \\
    &=\inp{u^{k+1}_{\gamma} - x^*}{\left(u^k - \frac{\alpha_{k,\gamma}}{\Gamma_k} \nabla f(y^{k+1}_{\gamma})\right) - u^{k}} +\inp{u^{k+1}_{\gamma} - x^*}{u^{k+1}_{\gamma} - \left(u^k - \frac{\alpha_{k,\gamma}}{\Gamma_k} \nabla f(y^{k+1}_{\gamma})\right)}.
  \end{align*}
  Using $u^{k+1}_{\gamma} = \textnormal{Proj}_{\bar{\mathcal{X}}} \left(u^k - \frac{\alpha_{k,\gamma}}{\Gamma_k} \nabla f(y^{k+1}_{\gamma})\right)$ and the projection property
    $\inp{\textnormal{Proj}_{\bar{\mathcal{X}}}(y) - x}{\textnormal{Proj}_{\bar{\mathcal{X}}}(y) - y} \leq 0$
  for all $y \in \R^d, x \in \bar{\mathcal{X}},$ we have
  \begin{align*}
    &-\inp{u^{k+1}_{\gamma} - x^*}{u^{k} - u^{k+1}_{\gamma}} \leq \inp{u^{k+1}_{\gamma} - x^*}{\left(u^k - \frac{\alpha_{k,\gamma}}{\Gamma_k} \nabla f(y^{k+1}_{\gamma})\right) - u^{k}} = -\inp{u^{k+1}_{\gamma} - x^*}{\frac{\alpha_{k,\gamma}}{\Gamma_k} \nabla f(y^{k+1}_{\gamma})}.
  \end{align*}
  Substituting to \eqref{eq:xrmEVqbJM},
  \begin{align*}
    &f(y^{k+1}_{\gamma}) - f(x^*) + \frac{(1 + \alpha_{k,\gamma}) \Gamma_{k+1, \gamma}}{2} \norm{u^{k+1}_{\gamma} - x^*}^2 - \left((f(y^{k}) - f(x^*)) + \frac{\Gamma_{k}}{2} \norm{u^{k} - x^*}^2\right) \\
    &= - (f(y^{k}) - f(y^{k+1}_{\gamma}) - \inp{\nabla f(y^{k+1}_{\gamma})}{y^{k} - y^{k+1}_{\gamma}}) \\
    &\quad + \inp{\nabla f(y^{k+1}_{\gamma})}{y^{k+1}_{\gamma} - y^{k}} \\
    &\quad + \frac{\Gamma_{k}}{2}\left(- \norm{u^{k} - u^{k+1}_{\gamma}}^2 - 2 \inp{u^{k+1}_{\gamma} - x^*}{\frac{\alpha_{k,\gamma}}{\Gamma_k} \nabla f(y^{k+1}_{\gamma})}\right) \\
    &= - (f(y^{k}) - f(y^{k+1}_{\gamma}) - \inp{\nabla f(y^{k+1}_{\gamma})}{y^{k} - y^{k+1}_{\gamma}}) \\
    &\quad + \inp{\nabla f(y^{k+1}_{\gamma})}{y^{k+1}_{\gamma} - y^{k}} \\
    &\quad - \frac{\Gamma_{k}}{2}\norm{u^{k} - u^{k+1}_{\gamma}}^2 \\
    &\quad - \alpha_{k,\gamma} \inp{u^{k+1}_{\gamma} - x^*}{\nabla f(y^{k+1}_{\gamma})} \\
    &= - (f(y^{k}) - f(y^{k+1}_{\gamma}) - \inp{\nabla f(y^{k+1}_{\gamma})}{y^{k} - y^{k+1}_{\gamma}}) \\
    &\quad + \inp{\nabla f(y^{k+1}_{\gamma})}{y^{k+1}_{\gamma} - y^{k} - \alpha_{k,\gamma} (u^{k+1}_{\gamma} - y^{k+1}_{\gamma})} \\
    &\quad - \frac{\Gamma_{k}}{2}\norm{u^{k} - u^{k+1}_{\gamma}}^2 \\
    &\quad - \alpha_{k,\gamma} \inp{y^{k+1}_{\gamma} - x^*}{\nabla f(y^{k+1}_{\gamma})}.
  \end{align*}
  In the last two equalities, we rearranged terms. Using the convexity of $f,$ we have $-(f(y^{k+1}_{\gamma}) - f(x^*)) \geq  - \inp{\nabla f(y^{k+1}_{\gamma})}{y^{k+1}_{\gamma} - x^*}$ and 
  \begin{align*}
    &(1 + \alpha_{k,\gamma})(f(y^{k+1}_{\gamma}) - f(x^*)) + \frac{(1 + \alpha_{k,\gamma}) \Gamma_{k+1, \gamma}}{2} \norm{u^{k+1}_{\gamma} - x^*}^2 - \left((f(y^{k}) - f(x^*)) + \frac{\Gamma_{k}}{2} \norm{u^{k} - x^*}^2\right) \\
    &\leq - (f(y^{k}) - f(y^{k+1}_{\gamma}) - \inp{\nabla f(y^{k+1}_{\gamma})}{y^{k} - y^{k+1}_{\gamma}}) \\
    &\quad + \inp{\nabla f(y^{k+1}_{\gamma})}{y^{k+1}_{\gamma} - y^{k} - \alpha_{k,\gamma} (u^{k+1}_{\gamma} - y^{k+1}_{\gamma})} \\
    &\quad - \frac{\Gamma_{k}}{2}\norm{u^{k} - u^{k+1}_{\gamma}}^2 \\
    &= - (f(y^{k}) - f(y^{k+1}_{\gamma}) - \inp{\nabla f(y^{k+1}_{\gamma})}{y^{k} - y^{k+1}_{\gamma}}) \\
    &\quad + \inp{\nabla f(y^{k+1}_{\gamma})}{(1 + \alpha_{k,\gamma})y^{k+1}_{\gamma} - y^{k} - \alpha_{k,\gamma} u^{k+1}_{\gamma}} \\
    &\quad - \frac{\Gamma_{k}}{2}\norm{u^{k} - u^{k+1}_{\gamma}}^2.
  \end{align*}
  In the last equality, we rearranged terms. Recall that
  \begin{align*}
    (1 + \alpha_{k,\gamma}) y^{k+1}_{\gamma} - y^{k} &= \alpha_{k,\gamma} u^{k} - \gamma \nabla f(y^{k}).
  \end{align*}
  Thus,
  \begin{align*}
    &(1 + \alpha_{k,\gamma})(f(y^{k+1}_{\gamma}) - f(x^*)) + \frac{(1 + \alpha_{k,\gamma}) \Gamma_{k+1, \gamma}}{2} \norm{u^{k+1}_{\gamma} - x^*}^2 - \left((f(y^{k}) - f(x^*)) + \frac{\Gamma_{k}}{2} \norm{u^{k} - x^*}^2\right) \\
    &\leq - (f(y^{k}) - f(y^{k+1}_{\gamma}) - \inp{\nabla f(y^{k+1}_{\gamma})}{y^{k} - y^{k+1}_{\gamma}}) \\
    &\quad + \alpha_{k,\gamma} \inp{\nabla f(y^{k+1}_{\gamma})}{u^{k} - u^{k+1}_{\gamma}} - \gamma \inp{\nabla f(y^{k+1}_{\gamma})}{\nabla f(y^{k})} \\
    &\quad - \frac{\Gamma_{k}}{2}\norm{u^{k} - u^{k+1}_{\gamma}}^2 \\
    &= - (f(y^{k}) - f(y^{k+1}_{\gamma}) - \inp{\nabla f(y^{k+1}_{\gamma})}{y^{k} - y^{k+1}_{\gamma}}) \\
    &\quad + \alpha_{k,\gamma} \inp{\nabla f(y^{k+1}_{\gamma})}{u^{k} - u^{k+1}_{\gamma}} \\
    &\quad - \frac{\gamma}{2} \norm{\nabla f(y^{k+1}_{\gamma})}^2 - \frac{\gamma}{2} \norm{\nabla f(y^{k})}^2 + \frac{\gamma}{2} \norm{\nabla f(y^{k+1}_{\gamma}) - \nabla f(y^{k})}^2 \\
    &\quad - \frac{\Gamma_{k}}{2}\norm{u^{k} - u^{k+1}_{\gamma}}^2,
  \end{align*}
  where we use $- \inp{a}{b} = \frac{1}{2}\norm{a - b}^2 - \frac{1}{2}\norm{a}^2 - \frac{1}{2}\norm{b}^2$ for all $a, b \in \R^d.$ Using Young's inequality,
  \begin{align*}
    &(1 + \alpha_{k,\gamma})(f(y^{k+1}_{\gamma}) - f(x^*)) + \frac{(1 + \alpha_{k,\gamma}) \Gamma_{k+1, \gamma}}{2} \norm{u^{k+1}_{\gamma} - x^*}^2 - \left((f(y^{k}) - f(x^*)) + \frac{\Gamma_{k}}{2} \norm{u^{k} - x^*}^2\right) \\
    &\leq - (f(y^{k}) - f(y^{k+1}_{\gamma}) - \inp{\nabla f(y^{k+1}_{\gamma})}{y^{k} - y^{k+1}_{\gamma}}) \\
    &\quad + \frac{\gamma}{2} \norm{\nabla f(y^{k+1}_{\gamma})}^2 + \frac{\alpha^2_{k, \gamma}}{2 \gamma}  \norm{u^{k} - u^{k+1}_{\gamma}}^2 \\
    &\quad - \frac{\gamma}{2} \norm{\nabla f(y^{k+1}_{\gamma})}^2 - \frac{\gamma}{2} \norm{\nabla f(y^{k})}^2 + \frac{\gamma}{2} \norm{\nabla f(y^{k+1}_{\gamma}) - \nabla f(y^{k})}^2 \\
    &\quad - \frac{\Gamma_{k}}{2}\norm{u^{k} - u^{k+1}_{\gamma}}^2 \\
    &= - (f(y^{k}) - f(y^{k+1}_{\gamma}) - \inp{\nabla f(y^{k+1}_{\gamma})}{y^{k} - y^{k+1}_{\gamma}}) \\
    &\quad  + \frac{\alpha^2_{k, \gamma}}{2 \gamma}  \norm{u^{k} - u^{k+1}_{\gamma}}^2 - \frac{\Gamma_{k}}{2}\norm{u^{k} - u^{k+1}_{\gamma}}^2 \\
    &\quad  - \frac{\gamma}{2} \norm{\nabla f(y^{k})}^2 + \frac{\gamma}{2} \norm{\nabla f(y^{k+1}_{\gamma}) - \nabla f(y^{k})}^2,
  \end{align*}
  where the terms $\frac{\gamma}{2} \norm{\nabla f(y^{k+1}_{\gamma})}^2$ are cancelled out. Since $\alpha_{k,\gamma} = \sqrt{\gamma \Gamma_k},$ the terms with $\norm{u^{k} - u^{k+1}_{\gamma}}$ are also cancelled out and 
  \begin{align}
    &(1 + \alpha_{k,\gamma})(f(y^{k+1}_{\gamma}) - f(x^*)) + \frac{(1 + \alpha_{k,\gamma}) \Gamma_{k+1, \gamma}}{2} \norm{u^{k+1}_{\gamma} - x^*}^2 - \left((f(y^{k}) - f(x^*)) + \frac{\Gamma_{k}}{2} \norm{u^{k} - x^*}^2\right) \nonumber \\
    &\leq - (f(y^{k}) - f(y^{k+1}_{\gamma}) - \inp{\nabla f(y^{k+1}_{\gamma})}{y^{k} - y^{k+1}_{\gamma}}) \nonumber \\
    &\quad  - \frac{\gamma}{2} \norm{\nabla f(y^{k})}^2 + \frac{\gamma}{2} \norm{\nabla f(y^{k+1}_{\gamma}) - \nabla f(y^{k})}^2 \nonumber \\ 
    &\leq - (f(y^{k}) - f(y^{k+1}_{\gamma}) - \inp{\nabla f(y^{k+1}_{\gamma})}{y^{k} - y^{k+1}_{\gamma}}) + \frac{\gamma}{2} \norm{\nabla f(y^{k+1}_{\gamma}) - \nabla f(y^{k})}^2, \label{eq:Sbvsyjk}
  \end{align}
  where the last inequality due to $\frac{\gamma}{2} \norm{\nabla f(y^{k})}^2 \geq 0.$ Using Lemma~\ref{lemma:smooth_convex}, we get 
  \begin{align*}
    &f(y^{k}) - f(y^{k+1}_{\gamma}) - \inp{\nabla f(y^{k+1}_{\gamma})}{y^{k} - y^{k+1}_{\gamma}} \\
    &\geq \norm{\nabla f(y^{k}) - \nabla f(y^{k + 1}_{\gamma})}^2 \int_{0}^{1} \frac{1 - v}{\ell(\norm{\nabla f(y^{k})} + \norm{\nabla f(y^{k}) - \nabla f(y^{k + 1}_{\gamma})} v)} d v \\
    &\geq \norm{\nabla f(y^{k}) - \nabla f(y^{k + 1}_{\gamma})}^2 \frac{1}{2 \ell(\norm{\nabla f(y^{k})} + \norm{\nabla f(y^{k}) - \nabla f(y^{k + 1}_{\gamma})})},
  \end{align*}
  where we use that $\ell$ is non-decreasing and bounded the term in the denominator by the maximum possible value with $v = 1.$
  Using triangle's inequality,
  \begin{align*}
    &f(y^{k}) - f(y^{k+1}_{\gamma}) - \inp{\nabla f(y^{k+1}_{\gamma})}{y^{k} - y^{k+1}_{\gamma}} \geq \norm{\nabla f(y^{k}) - \nabla f(y^{k + 1}_{\gamma})}^2 \frac{1}{2 \ell(2 \norm{\nabla f(y^{k})} + \norm{\nabla f(y^{k + 1}_{\gamma})})},
  \end{align*}
  Substituting to \eqref{eq:Sbvsyjk},
  \begin{align*}
    &(1 + \alpha_{k,\gamma})(f(y^{k+1}_{\gamma}) - f(x^*)) + \frac{(1 + \alpha_{k,\gamma}) \Gamma_{k+1, \gamma}}{2} \norm{u^{k+1}_{\gamma} - x^*}^2 - \left((f(y^{k}) - f(x^*)) + \frac{\Gamma_{k}}{2} \norm{u^{k} - x^*}^2\right) \nonumber \\
    &\leq \frac{1}{2} \left(\gamma - \frac{1}{\ell(2 \norm{\nabla f(y^{k})} + \norm{\nabla f(y^{k + 1}_{\gamma})})}\right) \norm{\nabla f(y^{k+1}_{\gamma}) - \nabla f(y^{k})}^2.
  \end{align*}
\end{proof}

\section{Convergence Theorems}

\subsection{Subquadratic and Quadratic Growth of $\ell$}

\begin{lemma}
  \label{lemma:bound_sub}
  Under Assumptions~\ref{ass:gen_smooth} and \ref{ass:convex}, let $\psi \,:\, \R_{+} \to \R_{+}$ such that $\psi(x) = \frac{x^2}{2 \ell(4 x)}$ is strictly increasing, $f(y) - f(x^*) \leq \delta$ for some $y \in \mathcal{X},$ and $\delta \in (0, \infty]$ such that $\ell\left(8 \sqrt{\delta \ell\left(0\right)}\right) \leq 2 \ell\left(0\right)$, then $\ell\left(4 \norm{\nabla f(y)}\right) \leq 2 \ell(0).$
\end{lemma}

\begin{proof}
  With this choice of $\delta,$ we get
  \begin{align*}
    \ell\left(4 \norm{\nabla f(y)}\right) \leq 2 \ell\left(0\right)
  \end{align*}
  because, due to $f(y) - f(x^*) \leq \delta$ and Lemma~\ref{lemma:func_to_grad_spec},
  \begin{align*}
    \ell\left(4 \norm{\nabla f(y^0)}\right) \leq \ell\left(4 \psi^{-1}(\delta)\right)
  \end{align*}
  and
  \begin{equation}
  \label{eq:ybjPYO}
  \begin{aligned}
    \ell\left(4 \psi^{-1}(\delta)\right) \leq 2 \ell\left(0\right) 
    &\Leftrightarrow \frac{\left(\psi^{-1}(\delta)\right)^2}{4 \ell\left(0\right)} \leq \frac{\left(\psi^{-1}(\delta)\right)^2}{2 \ell\left(4 \psi^{-1}(\delta)\right)} \overset{\psi(\psi^{-1}(\delta)) = \delta}{\Leftrightarrow} \frac{\left(\psi^{-1}(\delta)\right)^2}{4 \ell\left(0\right)} \leq \delta \Leftrightarrow \psi^{-1}(\delta) \leq 2 \sqrt{\delta \ell\left(0\right)} \\
    &\Leftrightarrow \delta \leq \psi\left(2 \sqrt{\delta \ell\left(0\right)}\right) \Leftrightarrow \delta \leq \frac{2 \delta \ell\left(0\right)}{\ell\left(8 \sqrt{\delta \ell\left(0\right)}\right)} \Leftrightarrow \ell\left(8 \sqrt{\delta \ell\left(0\right)}\right) \leq 2 \ell\left(0\right).
  \end{aligned}
  \end{equation}
\end{proof}

\THEORESUB*

\begin{proof}
  In our proof, we define the Lyapunov function $V_{k} \eqdef f(y^{k}) - f(x^*) + \frac{\Gamma_{k}}{2} \norm{u^{k} - x^*}^2.$ 
  After running \algname{GD}, we get $\ell\left(4 \norm{\nabla f(y^0)}\right) \leq 2 \ell(0)$ due to Lemma~\ref{lemma:bound_sub} and the choice of $\delta.$ Trivially, $V_{0} \leq V_{0}.$ Due to $f(y^{0}) - f(x^*) \leq \frac{\delta}{2}$ in Alg.~\ref{alg:main} and $\norm{y^{0} - x^*} \leq \norm{x^{0} - x^*}$ (\algname{GD} is monotonic; \citep{tyurin2024toward}[Lemma I.2]),
  \begin{equation}
  \begin{aligned}
    \label{eq:DrpsG}
    V_{0} 
    &= f(y^{0}) - f(x^*) + \frac{\Gamma_{0}}{2} \norm{y^{0} - x^*}^2 \leq \frac{\delta}{2} + \frac{\Gamma_{0}}{2} \norm{y^{0} - x^*}^2 \\
    &\leq \frac{\delta}{2} + \frac{\Gamma_{0}}{2} \norm{x^{0} - x^*}^2 \leq \delta
  \end{aligned}
  \end{equation}
  since $\Gamma_{0} = \frac{\delta}{\bar{R}^2}$ and $\bar{R} \geq \norm{x^{0} - x^*}.$ Using mathematical induction, we assume that 
  $\ell\left(4 \norm{\nabla f(y^k)}\right) \leq 2 \ell(0)$ and $V_{k} \leq \left(\prod_{i=0}^{k-1} \frac{1}{1 + \alpha_i}\right) V_0$ for some $k \geq 0.$

  Consider Lemma~\ref{thm:main_lemma} and the steps \eqref{eq:TqfxanGJGnaQsmS}. Then,
  \begin{align*}
    &(1 + \alpha_{k,\gamma})(f(y^{k+1}_{\gamma}) - f(x^*)) + \frac{(1 + \alpha_{k,\gamma}) \Gamma_{k+1, \gamma}}{2} \norm{u^{k+1}_{\gamma} - x^*}^2 - \left((f(y^{k}) - f(x^*)) + \frac{\Gamma_{k}}{2} \norm{u^{k} - x^*}^2\right) \nonumber \\
    &\leq \frac{1}{2} \left(\gamma - \frac{1}{\ell(2 \norm{\nabla f(y^{k})} + \norm{\nabla f(y^{k + 1}_{\gamma})})}\right) \norm{\nabla f(y^{k+1}_{\gamma}) - \nabla f(y^{k})}^2,
  \end{align*}
  where $0 \leq \gamma \leq \frac{1}{\ell(2 \norm{\nabla f(y^k)})}$ is a free parameter. Let us take the smallest $\gamma$ such that 
  \begin{align*}
    g(\gamma) \eqdef \gamma - \frac{1}{\ell(2 \norm{\nabla f(y^{k})} + \norm{\nabla f(y^{k + 1}_{\gamma})})} = 0
  \end{align*}
  and denote it as $\gamma^*.$ Such a choice exists because $g(\gamma)$ is continuous for all $\gamma \geq 0$ as a composition of continuous functions ($y^{k + 1}_{\gamma}$ is a continuous function of $\gamma$), $g(0) = - \frac{1}{\ell(2 \norm{\nabla f(y^{k})} + \norm{\nabla f(y^{k + 1}_{0})})} < 0,$ and 
  \begin{align*}
    g(\bar{\gamma}) = \bar{\gamma} - \frac{1}{\ell(2 \norm{\nabla f(y^{k})} + \norm{\nabla f(y^{k + 1}_{\bar{\gamma}})})} \geq \bar{\gamma} - \frac{1}{\ell(2 \norm{\nabla f(y^{k})})} = 0
  \end{align*}
  for $\bar{\gamma} = \frac{1}{\ell(2 \norm{\nabla f(y^{k})})}.$ Note that $\gamma^* \leq \frac{1}{\ell(2 \norm{\nabla f(y^{k})})}.$ For all $\gamma \leq \gamma^*,$ $g(\gamma) \leq 0$ and 
  \begin{equation}
  \label{eq:JGNJubw}
  \begin{aligned}
    &(1 + \alpha_{k,\gamma})(f(y^{k+1}_{\gamma}) - f(x^*)) + \frac{(1 + \alpha_{k,\gamma}) \Gamma_{k+1, \gamma}}{2} \norm{u^{k+1}_{\gamma} - x^*}^2\\
    &\leq (f(y^{k}) - f(x^*)) + \frac{\Gamma_{k}}{2} \norm{u^{k} - x^*}^2 =: V_k,
  \end{aligned}
  \end{equation}
  which ensures that 
  \begin{align*}
    &f(y^{k+1}_{\gamma}) - f(x^*) \leq V_k.
  \end{align*}
  Recall that $V_{k} \leq V_0 \overset{\eqref{eq:DrpsG}}{\leq} \delta.$ It means that 
  \begin{align*}
    &f(y^{k+1}_{\gamma}) - f(x^*) \leq \delta
  \end{align*}
  and 
  \begin{align*}
    \ell\left(4 \norm{\nabla f(y^{k+1}_{\gamma})}\right) \leq 2 \ell(0)
  \end{align*}
  for all $\gamma \leq \gamma^*$ due to Lemma~\ref{lemma:bound_sub}. Therefore, by the definition of $\gamma^*$ and using $\ell\left(4 \norm{\nabla f(y^{k})}\right) \leq 2 \ell(0),$
  \begin{align*}
    \gamma^* = \frac{1}{\ell(2 \norm{\nabla f(y^{k})} + \norm{\nabla f(y^{k + 1}_{\gamma^*})})} \geq \frac{1}{\max\{\ell(4 \norm{\nabla f(y^{k})}), \ell(4 \norm{\nabla f(y^{k + 1}_{\gamma^*})})\}} \geq \frac{1}{2 \ell(0)},
  \end{align*}
  meaning that we can take $\gamma = \frac{1}{2 \ell(0)}$ and \eqref{eq:JGNJubw} holds:
  \begin{align*}
    (1 + \alpha_{k,\gamma})(f(y^{k+1}_{\gamma}) - f(x^*)) + \frac{(1 + \alpha_{k,\gamma}) \Gamma_{k+1, \gamma}}{2} \norm{u^{k+1}_{\gamma} - x^*}^2 \leq V_k.
  \end{align*}
  Notice that $\alpha_{k,\gamma} = \alpha_k,$ $y^{k+1}_{\gamma} = y^{k+1},$ $\Gamma_{k+1, \gamma} = \Gamma_{k+1},$ and $u^{k+1}_{\gamma} = u^{k+1}$ with $\gamma = \frac{1}{2 \ell(0)}.$ Therefore, $(1 + \alpha_{k,\gamma})(f(y^{k+1}_{\gamma}) - f(x^*)) + \frac{(1 + \alpha_{k,\gamma}) \Gamma_{k+1, \gamma}}{2} \norm{u^{k+1}_{\gamma} - x^*}^2 = (1 + \alpha_{k}) V_{k+1},$
  \begin{align*}
    \ell\left(4 \norm{\nabla f(y^{k+1})}\right) \leq 2 \ell(0),
  \end{align*}
  and
  \begin{align*}
    V_{k+1} \leq \frac{1}{1 + \alpha_k} V_k \leq \left(\prod_{i=0}^{k} \frac{1}{1 + \alpha_i}\right) V_0,
  \end{align*}
  We have proved the next step of the induction. Finally, for all $k \geq 0,$
  \begin{align*}
    f(y^{k+1}) - f(x^*) 
    &\leq V_{k+1} \leq \left(\prod_{i=0}^{k} \frac{1}{1 + \alpha_i}\right) \left(f(y^{0}) - f(x^*) + \frac{\Gamma_{0}}{2} \norm{y^{0} - x^*}^2\right) \\
    &= \Gamma_{0} \left(\prod_{i=0}^{k} \frac{1}{1 + \alpha_i}\right) \left(\frac{1}{\Gamma_{0}}(f(y^{0}) - f(x^*)) + \frac{1}{2} \norm{y^{0} - x^*}^2\right).
  \end{align*}
  Since $f(y^{0}) - f(x^*) \leq \frac{\delta}{2},$ $\norm{y^{0} - x^*}^2 \leq \norm{x^{0} - x^*}^2 \leq \bar{R}^2,$ and $\Gamma_{k+1} = \Gamma_{0} \left(\prod_{i=0}^{k} \frac{1}{1 + \alpha_i}\right),$
  \begin{align*}
    f(y^{k+1}) - f(x^*) \leq \Gamma_{k+1} \left(\frac{\delta}{2 \Gamma_{0}} + \frac{1}{2} \bar{R}^2\right) = \Gamma_{k+1} \bar{R}^2
  \end{align*}
  because $\Gamma_{0} = \frac{\delta}{\bar{R}^2}.$ It is left to use Theorem~\ref{thm:gamma}.
\end{proof}

\THEORESUBRATE*

\begin{proof}
  At the beginning, we run \algname{GD}, which takes $k_{\textnormal{\algname{GD}}}(\delta)$ iterations (i.e., gradient evaluations). Next, using Theorem~\ref{thm:gamma} and the choice of $\gamma = \frac{1}{2 \ell\left(0\right)},$
  \begin{align*}
    \Gamma_{k+1} \leq \frac{18 \ell\left(0\right)}{\left(k + 1 - \bar{k}\right)^2}
  \end{align*}
  for all $k \geq \bar{k} \eqdef \max\left\{1 + \frac{1}{2} \log_{3/2}\left(\frac{\Gamma_0}{8 \ell\left(0\right)}\right), 0\right\}.$ Taking 
  \begin{align*}
    k \geq \frac{5 \sqrt{\ell\left(0\right)} \bar{R}}{\sqrt{\varepsilon}} + \bar{k},
  \end{align*}
  we get $f(y^{k+1}) - f(x^*) \leq \varepsilon$ due to Theorem~\ref{thm:convex_increasing}.
\end{proof}

\subsection{\sectionname}

\THEORESUBBETTER*

\begin{proof}
  In our proof, we define the Lyapunov function $V_{k} \eqdef f(y^{k}) - f(x^*) + \frac{\Gamma_{k}}{2} \norm{u^{k} - x^*}^2.$ 
  Trivially, $V_{0} \leq V_{0}$ and 
  \begin{align}
    \label{eq:LHzDkKZJSMGGGzeUTdc}
    V_{0} = f(y^{0}) - f(x^*) + \frac{\Gamma_{0}}{2} \norm{y^{0} - x^*}^2 \leq \Gamma_{0} R^2 \leq \Gamma_0 \bar{R}^2
  \end{align}
  when $\Gamma_{0} \geq \frac{2 (f(x^{0}) - f(x^*))}{\norm{x^{0} - x^*}^2} = \frac{2 (f(y^{0}) - f(x^*))}{\norm{y^{0} - x^*}^2}$ and $\bar{R} \geq R.$ Moreover,
  \begin{align*}
    f(y^{0}) - f(x^*) \leq \Gamma_{0} \bar{R}^2.
  \end{align*}
  Due to Lemma~\ref{lemma:func_to_grad_spec}, 
  \begin{align*}
    \norm{\nabla f(y^{0})} \leq \psi^{-1}(\Gamma_{0} \bar{R}^2).
  \end{align*}

  Using mathematical induction, we assume that 
  $\norm{\nabla f(y^{k})} \leq \psi^{-1}(\Gamma_{k} \bar{R}^2)$ and $V_{k} \leq \left(\prod_{i=0}^{k-1} \frac{1}{1 + \alpha_i}\right) V_0$ for some $k \geq 0.$

  Consider Lemma~\ref{thm:main_lemma} and the steps \eqref{eq:TqfxanGJGnaQsmS}. Then,
  \begin{align*}
    &(1 + \alpha_{k,\gamma})(f(y^{k+1}_{\gamma}) - f(x^*)) + \frac{(1 + \alpha_{k,\gamma}) \Gamma_{k+1, \gamma}}{2} \norm{u^{k+1}_{\gamma} - x^*}^2 - \left((f(y^{k}) - f(x^*)) + \frac{\Gamma_{k}}{2} \norm{u^{k} - x^*}^2\right) \nonumber \\
    &\leq \frac{1}{2} \left(\gamma - \frac{1}{\ell(2 \norm{\nabla f(y^{k})} + \norm{\nabla f(y^{k + 1}_{\gamma})})}\right) \norm{\nabla f(y^{k+1}_{\gamma}) - \nabla f(y^{k})}^2,
  \end{align*}
  where $0 \leq \gamma \leq \frac{1}{\ell(2 \norm{\nabla f(y^k)})}$ is a free parameter. Let us take the smallest $\gamma$ such that 
  \begin{align*}
    g(\gamma) \eqdef \gamma - \frac{1}{\ell(2 \norm{\nabla f(y^{k})} + \norm{\nabla f(y^{k + 1}_{\gamma})})} = 0
  \end{align*}
  and denote it as $\gamma^*$ (exists similarly to the proof of Theorem~\ref{thm:convex_increasing} and $\gamma^* \leq \frac{1}{\ell(2 \norm{\nabla f(y^k)})}$). For all $\gamma \leq \gamma^*,$ $g(\gamma) \leq 0$ and 
  \begin{equation}
  \label{eq:JGNJubwtwo}
  \begin{aligned}
    &(1 + \alpha_{k,\gamma})(f(y^{k+1}_{\gamma}) - f(x^*)) + \frac{(1 + \alpha_{k,\gamma}) \Gamma_{k+1, \gamma}}{2} \norm{u^{k+1}_{\gamma} - x^*}^2\\
    &\leq (f(y^{k}) - f(x^*)) + \frac{\Gamma_{k}}{2} \norm{u^{k} - x^*}^2 =: V_k.
  \end{aligned}
  \end{equation}
  Recall that 
  \begin{align*}
    V_{k} \leq \left(\prod_{i=0}^{k-1} \frac{1}{1 + \alpha_i}\right) V_0 = \frac{\Gamma_k}{\Gamma_0} V_0 \overset{\eqref{eq:LHzDkKZJSMGGGzeUTdc}}{\leq} \Gamma_k \bar{R}^2.
  \end{align*}
  Therefore,
  \begin{align*}
    &f(y^{k+1}_{\gamma}) - f(x^*) \overset{\eqref{eq:JGNJubwtwo}}{\leq} \frac{\Gamma_k \bar{R}^2}{1 + \alpha_{k,\gamma}}
  \end{align*}
  and 
  \begin{align}
    \label{eq:lwWFlsHXpktXywvzxd}
    \norm{\nabla f(y^{k+1}_{\gamma})} \leq \psi^{-1}\left(\frac{\Gamma_k \bar{R}^2}{1 + \alpha_{k,\gamma}} \right) \leq \psi^{-1}\left(\Gamma_k \bar{R}^2 \right)
  \end{align}
  for all $\gamma \leq \gamma^*$ due to Lemma~\ref{lemma:func_to_grad_spec}. Therefore, by the definition of $\gamma^*$ and using $\norm{\nabla f(y^{k})} \leq \psi^{-1}(\Gamma_{k} \bar{R}^2),$
  \begin{align*}
    \gamma^* = \frac{1}{\ell(2 \norm{\nabla f(y^{k})} + \norm{\nabla f(y^{k + 1}_{\gamma^*})})} \geq \frac{1}{\max\{\ell(4 \norm{\nabla f(y^{k})}), \ell(4 \norm{\nabla f(y^{k + 1}_{\gamma^*})})\}} \geq \frac{1}{\ell\left(4 \psi^{-1}\left(\Gamma_k \bar{R}^2 \right)\right)},
  \end{align*}
  meaning that we can take $\gamma_k = \frac{1}{\ell\left(4 \psi^{-1}\left(\Gamma_k \bar{R}^2 \right)\right)}$ and \eqref{eq:JGNJubw} holds:
  \begin{align*}
    (1 + \alpha_{k,\gamma})(f(y^{k+1}_{\gamma}) - f(x^*)) + \frac{(1 + \alpha_{k,\gamma}) \Gamma_{k+1, \gamma}}{2} \norm{u^{k+1}_{\gamma} - x^*}^2 \leq V_k.
  \end{align*}
  Notice that $\alpha_{k,\gamma} = \alpha_k,$ $y^{k+1}_{\gamma} = y^{k+1},$ $\Gamma_{k+1, \gamma} = \Gamma_{k+1},$ and $u^{k+1}_{\gamma} = u^{k+1}$ with $\gamma = \frac{1}{\ell\left(4 \psi^{-1}\left(\Gamma_k \bar{R}^2 \right)\right)}.$ Therefore, $(1 + \alpha_{k,\gamma})(f(y^{k+1}_{\gamma}) - f(x^*)) + \frac{(1 + \alpha_{k,\gamma}) \Gamma_{k+1, \gamma}}{2} \norm{u^{k+1}_{\gamma} - x^*}^2 = (1 + \alpha_{k}) V_{k+1},$
  \begin{align*}
    \norm{\nabla f(y^{k+1})} \overset{\eqref{eq:lwWFlsHXpktXywvzxd}}{\leq} \psi^{-1}\left(\frac{\Gamma_{k} \bar{R}^2}{1 + \alpha_{k}} \right) = \psi^{-1}\left(\Gamma_{k+1} \bar{R}^2 \right)
  \end{align*}
  and
  \begin{align*}
    V_{k+1} \leq \frac{1}{1 + \alpha_k} V_k \leq \left(\prod_{i=0}^{k} \frac{1}{1 + \alpha_i}\right) V_0,
  \end{align*}
  We have proved the next step of the induction. Finally, for all $k \geq 0,$
  \begin{align*}
    f(y^{k+1}) - f(x^*) 
    &\leq V_{k+1} \leq \Gamma_{k+1} \left(\frac{1}{\Gamma_0}(f(y^{0}) - f(x^*)) + \frac{1}{2} \norm{y^{0} - x^*}^2\right) \leq \Gamma_{k+1} \norm{x^{0} - x^*}^2
  \end{align*}
  because $\Gamma_{0} \geq \frac{2 (f(y^{0}) - f(x^*))}{\norm{y^{0} - x^*}^2},$ $\Gamma_{k+1} = \Gamma_{0} \left(\prod_{i=0}^{k} \frac{1}{1 + \alpha_i}\right),$ and $y^0 = x^0.$
\end{proof}

\THEORESUBBETTERRATE*

\begin{proof}
  Since $\gamma_k = 1 / \ell\left(4 \psi^{-1}\left(\Gamma_k \bar{R}^2 \right)\right) \geq \gamma_0 \eqdef 1 / \ell\left(4 \psi^{-1}\left(\Gamma_0 \bar{R}^2 \right)\right)$ for all $k \geq 0$ in Algorithm~\ref{alg:main_new}, and by Theorem~\ref{thm:gamma}, we conclude that 
  \begin{align}
    \label{eq:bbwfdsxPZiKSQFfK}
    \Gamma_{k} \leq \frac{9 \ell\left(4 \psi^{-1}\left(\Gamma_0 \bar{R}^2 \right)\right)}{\left(k - \bar{k}_1\right)^2}
  \end{align}
  for all $k > \bar{k}_1 \eqdef \max\left\{1 + \frac{1}{2} \log_{3/2}\left(\frac{\Gamma_0}{4 \ell\left(4 \psi^{-1}\left(\Gamma_0 \bar{R}^2 \right)\right)}\right), 0\right\}.$ 
  As in the proof of Lemma~\ref{lemma:bound_sub} (take $\delta = \Gamma_k \bar{R}^2$ in \eqref{eq:ybjPYO}):
  \begin{align}
    \label{eq:RdlkoZDiOPoLBgfaPwhF}
    \ell\left(4 \psi^{-1}\left(\Gamma_k \bar{R}^2 \right)\right) \leq 2 \ell(0) \Leftrightarrow \ell\left(8 \sqrt{\Gamma_k \bar{R}^2 \ell\left(0\right)}\right) \leq 2 \ell\left(0\right).
  \end{align}
  Let $k_{\textnormal{init}}$ be the smallest integer such that 
  \begin{align*}
    \ell\left(24 \sqrt{\frac{\ell\left(4 \psi^{-1}\left(\Gamma_0 \bar{R}^2 \right)\right) \ell\left(0\right) \bar{R}^2}{k_{\textnormal{init}}^2} }\right) \leq 2 \ell\left(0\right).
  \end{align*}
  Note that $k_{\textnormal{init}} < \infty,$ because $\ell$ is non-decreasing and continuous. Thus,
  \begin{align*}
    \ell\left(8 \sqrt{\Gamma_{k} \bar{R}^2 \ell\left(0\right)}\right) \leq 2 \ell\left(0\right)
  \end{align*}
  for all $k \geq k_{\textnormal{init}} + \bar{k}_1$ due to \eqref{eq:bbwfdsxPZiKSQFfK}, and $\gamma_k \geq \frac{1}{2 \ell(0)}$ for all $k \geq k_{\textnormal{init}} + \bar{k}_1$ due to \eqref{eq:RdlkoZDiOPoLBgfaPwhF}. We now repeat the previous arguments once again. Using Theorem~\ref{thm:gamma} with $\Gamma_0 \equiv \Gamma_{k_{\textnormal{init}} + \bar{k}_1}$, we conclude that 
  \begin{align*}
    \Gamma_{k+1 + k_{\textnormal{init}} + \bar{k}_1} \leq \frac{19 \ell(0)}{\left(k + 1 - \bar{k}\right)^2}
  \end{align*}
  for all $k \geq \bar{k} \eqdef \max\left\{1 + \frac{1}{2} \log_{3/2}\left(\frac{\Gamma_{k_{\textnormal{init}} + \bar{k}_1}}{8 \ell(0)}\right), 0\right\}.$ It is left to choose $k \geq \bar{k}$ such that 
  \begin{align*}
    \frac{19 \ell(0) R^2}{\left(k + 1 - \bar{k}\right)^2} \leq \varepsilon
  \end{align*}
  and use Theorem~\ref{thm:convex_increasing_better} to get the total oracle complexity
  \begin{align*}
  &\frac{5 \sqrt{\ell(0)} R}{\sqrt{\varepsilon}} + \max\left\{1 + \frac{1}{2} \log_{3/2}\left(\frac{\Gamma_{k_{\textnormal{init}} + \bar{k}_1}}{8 \ell(0)}\right), 0\right\} + k_{\textnormal{init}} + \max\left\{1 + \frac{1}{2} \log_{3/2}\left(\frac{\Gamma_0}{4 \ell\left(4 \psi^{-1}\left(\Gamma_0 \bar{R}^2 \right)\right)}\right), 0\right\} \\
  &\leq \frac{5 \sqrt{\ell(0)} R}{\sqrt{\varepsilon}} + k_{\textnormal{init}} + \max\left\{2 + \log_{3/2}\left(\frac{\Gamma_{0}}{4 \ell(0)}\right), 0\right\}
  \end{align*}
  because $\Gamma_k \leq \Gamma_0$ for all $k \geq 0$ and $\ell$ is non-decreasing.
\end{proof}

\subsubsection{Specialization for $(L_0, L_1)$--smoothness}
\THEORESUBBETTERRATEREFINE*

\begin{proof}
  Since $\psi(x) = \frac{x^2}{2 L_0 + 8 L_1 x},$ we get 
  \begin{align*}
  \psi^{-1}(t) = 4 L_1 t + \sqrt{16 L_1^2 t^2 + 2 L_0 t} \leq 8 L_1 t + \sqrt{2 L_0 t}
  \end{align*}
  for all $t \geq 0,$ and 
  \begin{align*}
    \gamma_k 
    &= \frac{1}{\ell\left(4 \psi^{-1}\left(\Gamma_k \bar{R}^2 \right)\right)} 
    \geq \frac{1}{L_0 + 4 L_1 (8 L_1 \Gamma_k \bar{R}^2  + \sqrt{2 L_0 \Gamma_k \bar{R}^2})} \\
    &= \frac{1}{L_0 + 32 L_1^2 \Gamma_k \bar{R}^2  + 4 L_1 \sqrt{2 L_0 \Gamma_k \bar{R}^2}} \overset{\textnormal{AM-GM}}{\geq} \frac{1}{2 L_0 + 48 L_1^2 \bar{R}^2 \Gamma_k}.
  \end{align*}
  Let $0 \leq k^* < \infty$ be the smallest $k$ such that $L_1^2 \bar{R}^2 \Gamma_k < L_0.$
  For all $k < k^*,$ we get $L_1^2 \bar{R}^2 \Gamma_k \geq L_0,$ $\gamma_{k} \geq \frac{1}{50 L_1^2 \bar{R}^2 \Gamma_{k}},$ and $\alpha_{k} \geq \frac{1}{8 L_1 \bar{R}}$ since $\Gamma_k$ is decreasing. Then,
  \begin{align*}
    \Gamma_{{k}+1} \leq \frac{\Gamma_{{k}}}{1 + \frac{1}{8 L_1 \bar{R}}}.
  \end{align*}
  for all $k < k^*.$
  We can unroll the recursion to get
  \begin{align}
    \label{eq:muQbfLvZyCGAqrixP}
    \Gamma_{k + 1} \leq \left(\frac{1}{1 + \frac{1}{8 L_1 \bar{R}}}\right)^{k + 1} \Gamma_0 \leq \exp\left(- \frac{k + 1}{8 L_1 \bar{R} + 1}\right) \Gamma_0.
  \end{align}
  for all $k < k^*.$ For all $k \geq k^*,$ $L_1^2 \bar{R}^2 \Gamma_k < L_0,$ $\gamma_k \geq \frac{1}{50 L_0},$ and can we use Theorem~\ref{thm:gamma} starting form the index $k^*:$
  \begin{align*}
    \Gamma_{k + k^*} \leq \frac{450 L_0}{\left(k - \bar{k}\right)^2}
  \end{align*}
  for all $k > \bar{k},$ where 
  \begin{align}
    \label{eq:rjSmRkLCAODUD}
    \bar{k} \eqdef \max\left\{1 + \frac{1}{2} \log_{3/2}\left(\frac{\Gamma_{k^*}}{200 L_0}\right), 0\right\} \leq \max\left\{1 + \frac{1}{2} \log_{3/2}\left(\frac{\Gamma_{0}}{200 L_0}\right), 0\right\},
  \end{align}
  where the first inequality due to $\Gamma_{k^*} \leq \Gamma_0.$ If $k^* = 0,$ then
  \begin{align*}
    \Gamma_{k} \leq \frac{450 L_0}{\left(k - \bar{k}\right)^2}
  \end{align*}
  for all $k > \bar{k}.$ If $k^* > 0,$ then
  \begin{align*}
    \frac{L_0}{L_1^2 \bar{R}^2} \leq \Gamma_{k^* - 1} \overset{\eqref{eq:muQbfLvZyCGAqrixP}}{\leq} \exp\left(- \frac{k^* - 1}{8 L_1 \bar{R} + 1}\right) \Gamma_0
  \end{align*}
  and 
  \begin{align*}
    k^* \leq 1 + \left(8 L_1 \bar{R} + 1\right)\log\left(\frac{L_1^2 \bar{R}^2 \Gamma_0}{L_0}\right).
  \end{align*}
  In total,
  \begin{align}
    \label{eq:npWuFLeurMR}
    k^* \leq \max\left\{1 + \left(8 L_1 \bar{R} + 1\right)\log\left(\frac{L_1^2 \bar{R}^2 \Gamma_0}{L_0}\right), 0\right\}.
  \end{align}
  There are two main regimes of $\Gamma_{k}$. The first regime is
  \begin{align}
    \label{eq:dtowaBzwKvS}
    \Gamma_{k} \leq \frac{450 L_0}{\left(k - (\bar{k} + k^*)\right)^2}
  \end{align}
  for all $k > \bar{k} + k^*,$ and for all
  \begin{align*}
    k \geq \max\left\{1 + \left(8 L_1 \bar{R} + 1\right)\log\left(\frac{L_1^2 \bar{R}^2 \Gamma_0}{L_0}\right), 0\right\} + \max\left\{2 + 3\log\left(\frac{\Gamma_{0}}{200 L_0}\right), 0\right\},
  \end{align*}
  due to \eqref{eq:rjSmRkLCAODUD} and \eqref{eq:npWuFLeurMR}.
  The second regime is 
  \begin{align}
    \label{eq:BozWiZ}
    \Gamma_{k} \leq \exp\left(- \frac{k}{8 L_1 \bar{R} + 1}\right) \Gamma_0
  \end{align}
  for all $k \leq k^*$ due to \eqref{eq:muQbfLvZyCGAqrixP}. 
  
  Using Theorem~\ref{thm:convex_increasing_better},
  \begin{align*}
    f(y^{k+1}) - f(x^*) \leq \Gamma_{k+1} R^2.
  \end{align*}
  If $\frac{L_1^2 \bar{R}^2 \Gamma_0}{L_0} \leq \frac{\Gamma_0 R^2}{\varepsilon},$
  then $f(y^{k+1}) - f(x^*) \leq \varepsilon$ after
  \begin{align*}
    \cO\left(\frac{\sqrt{L_0} R}{\sqrt{\varepsilon}} + \max\left\{\left(L_1 \bar{R} + 1\right)\log\left(\frac{L_1^2 \bar{R}^2 \Gamma_0}{L_0}\right), 0\right\} + \max\left\{\log\left(\frac{\Gamma_{0}}{L_0}\right), 0\right\}\right)
  \end{align*}
  iterations due to \eqref{eq:dtowaBzwKvS}. If $\frac{L_1^2 \bar{R}^2 \Gamma_0}{L_0} > \frac{\Gamma_0 R^2}{\varepsilon}$ and $k^* > (8 L_1 \bar{R} + 1) \log\left((\Gamma_0 R^2) / \varepsilon\right),$ then $f(y^{k+1}) - f(x^*) \leq \varepsilon$ after
  \begin{align*}
    \cO\left((L_1 \bar{R} + 1) \log\left(\frac{\Gamma_0 R^2}{\varepsilon}\right)\right)
  \end{align*}
  iterations due to \eqref{eq:BozWiZ}. If $\frac{L_1^2 \bar{R}^2 \Gamma_0}{L_0} > \frac{\Gamma_0 R^2}{\varepsilon}$ and $k^* \leq (8 L_1 \bar{R} + 1) \log\left((\Gamma_0 R^2) / \varepsilon\right),$ then $f(y^{k+1}) - f(x^*) \leq \varepsilon$ after
  \begin{align*}
    \cO\left(\frac{\sqrt{L_0} R}{\sqrt{\varepsilon}} + (L_1 \bar{R} + 1) \log\left(\frac{\Gamma_0 R^2}{\varepsilon}\right) + \max\left\{\log\left(\frac{\Gamma_{0}}{L_0}\right), 0\right\}\right)
  \end{align*}
  iterations due to \eqref{eq:dtowaBzwKvS}. It left to combine all cases.
\end{proof}

\subsection{Superquadratic Growth of $\ell$}

\THEORESUPER*

\begin{proof}
  In our proof, we define the Lyapunov function $V_{k} \eqdef f(y^{k}) - f(x^*) + \frac{\Gamma_{k}}{2} \norm{u^{k} - x^*}^2.$ 
  
  \emph{(Base case:)} Clearly, $\norm{u^0 - x^*} = \norm{y^0 - x^*} \leq \norm{x^0 - x^*} \leq 2 \bar{R}$ due the the monotonicity of \algname{GD} \citep{tyurin2024toward}[Lemma I.2] and $\bar{R} \geq R.$ Moreover, $V_{0} \leq V_{0}$ and 
  \begin{equation}
  \begin{aligned}
    \label{eq:WQjOO}
    f(y^{0}) - f(x^*) \leq V_{0} 
    &= f(y^{0}) - f(x^*) + \frac{\Gamma_{0}}{2} \norm{y^{0} - x^*}^2 \leq \frac{\delta}{2} + \frac{\Gamma_{0}}{2} \norm{y^{0} - x^*}^2 \\
    &\leq \frac{\delta}{2} + \frac{\Gamma_{0}}{2} \norm{x^{0} - x^*}^2 \leq \delta \leq \Delta
  \end{aligned}
  \end{equation}
  since $\Gamma_{0} = \frac{\delta}{\bar{R}^2},$ $\bar{R} \geq \norm{x^{0} - x^*},$ and $\delta \leq \Delta.$
  
  Thus,
  \begin{align*}
    \norm{\nabla f(y^0)} \leq \max_{f(x) - f(x^*) \leq \Delta, \norm{x - x^*} \leq 2 \bar{R}} \norm{\nabla f(x)} \leq M_{\bar{R}}.
  \end{align*}
  Using Lemma~\ref{lemma:func_to_grad_spec_diff}, either $\norm{\nabla f(y^0)} \leq \Delta_{\textnormal{left}}(\delta)$ or $\norm{\nabla f(y^0)} \geq \Delta_{\textnormal{right}}(\delta).$ However, the latter is not possible because $\Delta_{\textnormal{right}}(\delta) > M_{\bar{R}}$ and $\norm{\nabla f(y^0)} \leq M_{\bar{R}}.$ Thus, $\ell\left(4 \norm{\nabla f(y^0)}\right) \leq \ell(4 \Delta_{\textnormal{left}}(\delta)) \leq 2 \ell(0),$
  where the last inequality due to the conditions of the theorem.
  
  Using mathematical induction, we assume that 
  $\ell\left(4 \norm{\nabla f(y^k)}\right) \leq 2 \ell(0),$
  \begin{align}
    \label{eq:hMMXuAoUGQxLgJoaa}
    V_{k} \leq \left(\prod_{i=0}^{k-1} \frac{1}{1 + \alpha_i}\right) V_0,
  \end{align}
  $\norm{u^k - x^*} \leq 2 \bar{R},$ and $\norm{y^k - x^*} \leq 2 \bar{R}$ for some $k \geq 0$ (the base case has been proved in the previous steps).

  Consider Lemma~\ref{thm:main_lemma} and the steps \eqref{eq:TqfxanGJGnaQsmS}. Then,
  \begin{align*}
    &(1 + \alpha_{k,\gamma})(f(y^{k+1}_{\gamma}) - f(x^*)) + \frac{(1 + \alpha_{k,\gamma}) \Gamma_{k+1, \gamma}}{2} \norm{u^{k+1}_{\gamma} - x^*}^2 - \left((f(y^{k}) - f(x^*)) + \frac{\Gamma_{k}}{2} \norm{u^{k} - x^*}^2\right) \nonumber \\
    &\leq \frac{1}{2} \left(\gamma - \frac{1}{\ell(2 \norm{\nabla f(y^{k})} + \norm{\nabla f(y^{k + 1}_{\gamma})})}\right) \norm{\nabla f(y^{k+1}_{\gamma}) - \nabla f(y^{k})}^2,
  \end{align*}
  where $0 \leq \gamma \leq \frac{1}{\ell(2 \norm{\nabla f(y^k)})}$ is a free parameter. Let us take the smallest $\gamma$ such that 
  \begin{align*}
    g(\gamma) \eqdef \gamma - \frac{1}{\ell(2 \norm{\nabla f(y^{k})} + \norm{\nabla f(y^{k + 1}_{\gamma})})} = 0
  \end{align*}
  and denote it as $\gamma^*$ (exists similarly to the proof of Theorem~\ref{thm:convex_increasing} and $\gamma^* \leq \frac{1}{\ell(2 \norm{\nabla f(y^k)})}$). 
  For all $\gamma \leq \gamma^*,$ $g(\gamma) \leq 0$ and 
  \begin{equation}
  \label{eq:WCsqvicNrlbLMmi}
  \begin{aligned}
    &(1 + \alpha_{k,\gamma})(f(y^{k+1}_{\gamma}) - f(x^*)) + \frac{(1 + \alpha_{k,\gamma}) \Gamma_{k+1, \gamma}}{2} \norm{u^{k+1}_{\gamma} - x^*}^2\\
    &\leq (f(y^{k}) - f(x^*)) + \frac{\Gamma_{k}}{2} \norm{u^{k} - x^*}^2 =: V_k,
  \end{aligned}
  \end{equation}
  which ensures that 
  \begin{align}
    \label{eq:GQhgMNQ}
    &f(y^{k+1}_{\gamma}) - f(x^*) \leq V_k \overset{\eqref{eq:hMMXuAoUGQxLgJoaa}}{\leq} V_0 \overset{\eqref{eq:WQjOO}}{\leq} \delta \leq \Delta.
  \end{align}
  Moreover, due to \eqref{eq:WCsqvicNrlbLMmi} and \eqref{eq:TqfxanGJGnaQsmS}, we have
  \begin{align*}
    \frac{\Gamma_{k}}{2} \norm{u^{k+1}_{\gamma} - x^*}^2 
    &= \frac{(1 + \alpha_{k,\gamma}) \Gamma_{k+1, \gamma}}{2} \norm{u^{k+1}_{\gamma} - x^*}^2 \leq V_k \\
    &\overset{\eqref{eq:hMMXuAoUGQxLgJoaa}}{\leq} \left(\prod_{i=0}^{k-1} \frac{1}{1 + \alpha_i}\right) V_0 = \frac{\Gamma_k}{\Gamma_0} \left((f(y^{0}) - f(x^*)) + \frac{\Gamma_{0}}{2} \norm{u^{0} - x^*}^2\right) \\
    &\overset{\textnormal{Alg.~\ref{alg:main}}}{\leq} \Gamma_k \left(\frac{\delta}{2 \Gamma_0} + \frac{1}{2} \norm{u^{0} - x^*}^2\right) \leq \Gamma_k \bar{R}^2,
  \end{align*}
  where the last inequality due to $\Gamma_{0} = \frac{\delta}{\bar{R}^2}$ and $\norm{u^{0} - x^*}^2 \leq \bar{R}^2.$ Thus,
  \begin{align}
    \label{eq:uqPpcQwSacrWB}
    \norm{u^{k+1}_{\gamma} - x^*}^2 \leq 2 \bar{R}
  \end{align}
  for all $\gamma \leq \gamma^*.$ Now, consider $y^{k+1}_{\gamma}$ from \eqref{eq:TqfxanGJGnaQsmS}:
  \begin{equation}
  \label{eq:sTkqnDdGuBzMyOyjwlAE}
  \begin{aligned}
    &\norm{y^{k+1}_{\gamma} - x^*} \\
    &= \norm{\frac{1}{1 + \alpha_{k,\gamma}} y^{k} + \frac{\alpha_{k,\gamma}}{1 + \alpha_{k,\gamma}} u^{k} - \frac{\gamma}{1 + \alpha_{k,\gamma}} \nabla f(y^{k}) - x^*} \\
    &= \norm{\frac{1}{1 + \alpha_{k,\gamma}} \left(\left(y^{k} - \gamma \nabla f(y^{k})\right) - x^*\right) + \frac{\alpha_{k,\gamma}}{1 + \alpha_{k,\gamma}} (u^{k} -  x^*)} \\
    &\leq \frac{1}{1 + \alpha_{k,\gamma}} \norm{\left(y^{k} - \gamma \nabla f(y^{k})\right) - x^*} + \frac{\alpha_{k,\gamma}}{1 + \alpha_{k,\gamma}} \norm{u^{k} -  x^*},
  \end{aligned}
  \end{equation}
  where we use Triangle's inequality. Notice that 
  \begin{align}
    \label{eq:FAJvbtdJQHTSpIXPGE}
    \gamma \leq \frac{1}{\ell(2 \norm{\nabla f(y^{k})})}
  \end{align}
  for all $\gamma \leq \gamma^*$ because $\gamma^* \leq \frac{1}{\ell(2 \norm{\nabla f(y^{k})})}.$ Thus,
  \begin{align*}
    &\norm{\left(y^{k} - \gamma \nabla f(y^{k})\right) - x^*}^2 
    = \norm{y^{k} - x^*}^2 - 2 \gamma \inp{y^{k} - x^*}{\nabla f(y^{k})} + \gamma^2 \norm{\nabla f(y^{k})}^2\\
    &\overset{\textnormal{L.\,\ref{lemma:smooth_convex}}}{\leq} \norm{y^{k} - x^*}^2 + 2 \gamma \left(f(x^*) - f(y^k) - \norm{\nabla f(y^k)}^2 \int_{0}^{1} \frac{1 
    - v}{\ell(\norm{\nabla f(y^k)} v)} d v\right) + \gamma^2 \norm{\nabla f(y^{k})}^2\\
    &\leq \norm{y^{k} - x^*}^2 + \gamma \norm{\nabla f(y^{k})}^2 \left(\gamma - 2 \int_{0}^{1} \frac{1 
    - v}{\ell(\norm{\nabla f(y^k)} v)} d v\right).
  \end{align*}
  In the last inequality, we use $f(x^*) - f(y^k) \leq 0.$ Next,
  \begin{align*}
    \norm{\left(y^{k} - \gamma \nabla f(y^{k})\right) - x^*}^2 
    &\overset{\eqref{eq:FAJvbtdJQHTSpIXPGE}}{\leq} \norm{y^{k} - x^*}^2 + \gamma \norm{\nabla f(y^{k})}^2 \left(\frac{1}{\ell(2 \norm{\nabla f(y^{k})})} - 2 \int_{0}^{1} \frac{1 
    - v}{\ell(\norm{\nabla f(y^k)} v)} d v\right) \\
    &\leq \norm{y^{k} - x^*}^2 + \gamma \norm{\nabla f(y^{k})}^2 \left(\frac{1}{\ell(2 \norm{\nabla f(y^{k})})} - \frac{1}{\ell(\norm{\nabla f(y^k)})}\right) \\
    &\leq \norm{y^{k} - x^*}^2
  \end{align*}
  because $\ell$ is non-decreasing. Thus, by the induction assumption, $\norm{\left(y^{k} - \gamma \nabla f(y^{k})\right) - x^*} \leq \norm{y^{k} - x^*} \leq 2 \bar{R},$ $\norm{u^{k} -  x^*} \leq 2 \bar{R},$ and
  \begin{align}
    \label{eq:PzTfxxvAahIZEXWcHRe}
    \norm{y^{k+1}_{\gamma} - x^*} \leq 2 \bar{R}
  \end{align}
  for all $\gamma \leq \gamma^*,$ due to \eqref{eq:sTkqnDdGuBzMyOyjwlAE}.

  Using the last inequality and \eqref{eq:GQhgMNQ},
  \begin{align*}
    \norm{\nabla f(y^{k+1}_{\gamma})} \leq \max_{f(x) - f(x^*) \leq \Delta, \norm{x - x^*} \leq 2 \bar{R}} \norm{\nabla f(x)} \leq M_{\bar{R}}.
  \end{align*}
  Using \eqref{eq:GQhgMNQ} and Lemma~\ref{lemma:func_to_grad_spec_diff}, either $\norm{\nabla f(y^{k+1}_{\gamma})} \leq \Delta_{\textnormal{left}}(\delta)$ or $\norm{\nabla f(y^{k+1}_{\gamma})} \geq \Delta_{\textnormal{right}}(\delta).$ However, the latter is not possible because $\Delta_{\textnormal{right}}(\delta) > M_{\bar{R}}$ and $\norm{\nabla f(y^{k+1}_{\gamma})} \leq M_{\bar{R}}.$ Thus, 
  \begin{align}
    \label{eq:YfdEfhm}
    \ell\left(4 \norm{\nabla f(y^{k+1}_{\gamma})}\right) \leq \ell(4 \Delta_{\textnormal{left}}(\delta)) \leq 2 \ell(0).
  \end{align}
  Therefore, by the definition of $\gamma^*$ and using $\ell\left(4 \norm{\nabla f(y^{k})}\right) \leq 2 \ell(0),$
  \begin{align*}
    \gamma^* = \frac{1}{\ell(2 \norm{\nabla f(y^{k})} + \norm{\nabla f(y^{k + 1}_{\gamma^*})})} \geq \frac{1}{\max\{\ell(4 \norm{\nabla f(y^{k})}), \ell(4 \norm{\nabla f(y^{k + 1}_{\gamma^*})})\}} \geq \frac{1}{2 \ell(0)},
  \end{align*}
  meaning that we can take $\gamma = \frac{1}{2 \ell(0)}$ and \eqref{eq:WCsqvicNrlbLMmi} holds:
  \begin{align*}
    (1 + \alpha_{k,\gamma})(f(y^{k+1}_{\gamma}) - f(x^*)) + \frac{(1 + \alpha_{k,\gamma}) \Gamma_{k+1, \gamma}}{2} \norm{u^{k+1}_{\gamma} - x^*}^2 \leq V_k.
  \end{align*}
  Notice that $\alpha_{k,\gamma} = \alpha_k,$ $y^{k+1}_{\gamma} = y^{k+1},$ $\Gamma_{k+1, \gamma} = \Gamma_{k+1},$ and $u^{k+1}_{\gamma} = u^{k+1}$ with $\gamma = \frac{1}{2 \ell(0)}.$ Therefore, $(1 + \alpha_{k,\gamma})(f(y^{k+1}_{\gamma}) - f(x^*)) + \frac{(1 + \alpha_{k,\gamma}) \Gamma_{k+1, \gamma}}{2} \norm{u^{k+1}_{\gamma} - x^*}^2 = (1 + \alpha_{k}) V_{k+1},$
  \begin{align*}
    \ell\left(4 \norm{\nabla f(y^{k+1})}\right) \overset{\eqref{eq:YfdEfhm}}{\leq} 2 \ell(0),
  \end{align*}
  \begin{align*}
    V_{k+1} \leq \frac{1}{1 + \alpha_k} V_k \leq \left(\prod_{i=0}^{k} \frac{1}{1 + \alpha_i}\right) V_0,
  \end{align*}
  \begin{align*}
    \norm{u^{k+1} - x^*}^2 \overset{\eqref{eq:uqPpcQwSacrWB}}{\leq} 2 \bar{R},
  \end{align*}
  and
  \begin{align*}
    \norm{y^{k+1} - x^*} \overset{\eqref{eq:PzTfxxvAahIZEXWcHRe}}{\leq} 2 \bar{R}.
  \end{align*}
  
  We have proved the next step of the induction. Finally, for all $k \geq 0,$
  \begin{align*}
    f(y^{k+1}) - f(x^*) 
    &\leq V_{k+1} \leq \left(\prod_{i=0}^{k} \frac{1}{1 + \alpha_i}\right) \left(f(y^{0}) - f(x^*) + \frac{\Gamma_{0}}{2} \norm{y^{0} - x^*}^2\right) \\
    &\leq \Gamma_{0} \left(\prod_{i=0}^{k} \frac{1}{1 + \alpha_i}\right) \left(\frac{\delta}{2 \Gamma_{0}} + \frac{1}{2} \norm{y^{0} - x^*}^2\right) 
    \leq \Gamma_{k+1} \bar{R}^2
  \end{align*}
  because \algname{GD} by \citep{tyurin2024toward}[Lemma I.2] returns $\bar{x} = y^0$ such that $\norm{y^0 - x^*} \leq \norm{x^0 - x^*} \leq \bar{R}.$ Moreover, we use $\Gamma_{0} = \frac{\delta}{\bar{R}^2}$ and $\Gamma_{k+1} = \Gamma_{0} \left(\prod_{i=0}^{k} \frac{1}{1 + \alpha_i}\right).$ It is left to use Theorem~\ref{thm:gamma}.
\end{proof}

\THEORESUPERRATE*

\begin{proof}
  The proof of this theorem repeats the proof of Theorem~\ref{thm:convex_increasing_rate}, with the only change being that the conditions on $\delta$ are different.
\end{proof}

\end{document}